\documentclass[11pt]{amsart}

\date{version of \today}
\usepackage{amsmath}
\usepackage{amssymb}
\usepackage{amsxtra}
\usepackage{abstract}
\usepackage{mathrsfs}
\usepackage{mathtools}
\usepackage[all]{xy}
\usepackage{amscd}
\usepackage{amsthm}
\usepackage{graphicx}

\usepackage[normalem]{ulem}
\usepackage{color}
\usepackage{appendix}
\usepackage{tikz}
\usepackage{tikz-cd}
\usepackage{url}

\newtheorem{Thm}{Theorem}[section]
\newtheorem{Lem}[Thm]{Lemma}
\newtheorem{Def}[Thm]{Definition}
\newtheorem{Cor}[Thm]{Corollary}
\newtheorem{Prop}[Thm]{Proposition}
\newtheorem{Ex1}[Thm]{Example}
\newtheorem{Rem1}[Thm]{Remark}
\newtheorem{Conj}[Thm]{Conjecture}
\newtheorem{Prob}[Thm]{Problem}

\usepackage{bbm}
\usepackage{enumitem}

\tikzset{every picture/.style={line width=0.75pt}} %set default line width to 0.75pt

\newcommand{\sHom}{\underline{\mathrm{Hom}}}
\newcommand{\Hom}{\mathrm{Hom}}
\renewcommand{\top}{\mathop{\mathrm{top}}}
\newcommand{\rad}{\operatorname{rad}}
\newcommand{\soc}{\mathop{\mathrm{soc}}}

\newenvironment{theorem}{\begin{Thm}}{\end{Thm}}
\newenvironment{lemma}{\begin{Lem}}{\end{Lem}}
\newenvironment{definition}{\begin{Def}}{\end{Def}}

\newenvironment{proposition}{\begin{Prop}}{\end{Prop}}
\newenvironment{example}{\begin{Ex1}}{\end{Ex1}}
\newenvironment{remark}{\begin{Rem1}}{\end{Rem1}}

\usepackage[bookmarks=true, colorlinks=true, citecolor=blue, linkcolor=red]{hyperref}
\newcommand{\qa}{\Bbbk Q/I}

\newcommand{\m}{m}

\newcommand{\lra}{\longrightarrow}

\newcommand{\ra}{\rightarrow}
\newcommand{\sdp}{\times\kern-.2em\vrule height1.1ex depth-.05ex}
\newcommand{\epi}{\lra \kern-.8em\ra}

 \normalbaselines
\begin{document}

\title{Brauer graph algebras are closed under stable equivalence of Morita type} \author{Pengyun Chen$^a$, Nengqun Li$^b$, Yuming Liu$^a$, and Bohan Xing$^{a,*}$}
	\maketitle
	
	\renewcommand{\thefootnote}{\alph{footnote}}
	\setcounter{footnote}{-1} 
	\renewcommand{\thefootnote}{\alph{footnote}}
	\setcounter{footnote}{-1} 
	\setcounter{footnote}{-1} \footnote{$^a$Pengyun Chen, Yuming Liu, and Bohan Xing, School of Mathematical Sciences, Laboratory of Mathematics and Complex Systems, Beijing Normal University,
Beijing 100875,  P. R. China.}
	\setcounter{footnote}{-1} \footnote{$^b$Nengqun Li, School of Mathematics, Liaoning Normal University,
Dalian 116029,  P. R. China.}
	\setcounter{footnote}{-1} \footnote{E-mail addresses: pychen@mail.bnu.edu.cn (P. Chen); linengqun@lnnu.edu.cn (N. Li); ymliu@bnu.edu.cn (Y. Liu); bhxing@mail.bnu.edu.cn (B. Xing).}
	\setcounter{footnote}{-1} \footnote{$^*$Corresponding author.}
	
\noindent{\bf Abstract:} We study stable equivalences of Brauer graph algebras. We prove that a basic symmetric algebra without semisimple summands which is stably equivalent to a Brauer graph algebra is again a Brauer graph algebra. In particular, Brauer graph algebras are closed up to semisimple summands under stable equivalence of Morita type. As consequences, we reprove the closure under derived equivalence established by Antipov and Zvonareva and obtain a solution of the reconstruction problem posed by Rickard and Rouquier for algebras stably equivalent of Morita type to Brauer graph algebras.

\setcounter{tocdepth}{2}
\tableofcontents

\section{Introduction}\label{sec:1}

As one of the basic equivalences, the stable equivalence plays an important role in the representation theory of finite dimensional algebras and finite groups. One early result on stable equivalence arising from modular group representation theory says that any group algebra of finite representation type is stably equivalent to a Nakayama algebra. Moreover, it is well-known that the symmetric algebras stably equivalent to symmetric Nakayama algebras are exactly those given by Brauer trees (see for example in the textbook \cite{ARS}).

Brauer graph algebras are generalizations of Brauer tree algebras and they coincide with symmetric special biserial algebras over an algebraically closed field. In fact, the Brauer graph algebras of finite representation type are precisely Brauer tree algebras. Each Brauer graph algebra is defined by a Brauer graph, which is a ribbon graph endowed with a multiplicity function, and the representation theory information of a Brauer graph algebra is encoded in its Brauer graph. In particular, each Brauer graph algebra can be defined by quiver with relations. For a nice survey on Brauer graph algebras, we refer to Schroll \cite{S}.

According to a result claimed by Pogorzaly in \cite{P}, self-injective special biserial algebras can be stably equivalent to stably biserial algebras only, and these two classes coincide. However, Ariki, Iijima and Park showed in \cite[Example A.7]{AIP} that these two classes do not coincide. By modifying some ideas from \cite{P}, Antipov and Zvonareva \cite{AZ1} gave a complete proof for the facts that a self-injective special biserial algebra can be stably equivalent to a stably biserial algebra only (here we assume that the self-injective algebras have no summand isomorphic to a Nakayama algebra of radical square zero) and that the Auslander-Reiten conjecture is true for algebras stably equivalent to special biserial algebras. Moreover, they showed that each symmetric stably biserial algebra is associated with a Brauer graph $(\Gamma,m)$ together with a set $\mathcal{L}$ of loops, called deformed loops; and in characteristic different from $2$, symmetric special biserial algebras (that is, Brauer graph algebras) coincide with symmetric stably biserial algebras.

Recently, Antipov and Zvonareva \cite{AZ2} proved, using some subtle derived invariants, that Brauer graph algebras are closed under derived equivalence. More recently, Opper and Zvonareva \cite{OZ} gave a complete classification of Brauer graph algebras up to derived equivalence. Nevertheless, the question of whether Brauer graph algebras are closed under stable equivalence remains largely unexplored. In general, this question may have a negative answer. For example, the Brauer graph algebra $\Bbbk[x]/(x^2)$ is stably equivalent to the hereditary algebra $\Bbbk(\cdot\rightarrow\cdot)$. However, if we exclude the algebras of radical square zero, then we even do not know whether a Brauer graph algebra can be stably equivalent to some non-symmetric algebra. In fact, when we impose the symmetric condition, we obtain the following result.

\begin{Thm}\textnormal{(see Theorem \ref{thm:st-BGA})}\label{thm:1.1}
Let $A$ be a Brauer graph algebra, and let $B$ be a basic symmetric algebra with no semisimple summands. If $B$ and $A$ are stably equivalent, then $B$ is also a Brauer graph algebra.
\end{Thm}

In order to prove the above result, we shall use (and generalize) some invariants of the (triangulated) stable categories of Brauer graph algebras which are proved by Antipov in \cite{A}. For example, for a Brauer graph algebra $A$, he shows that
$$
\mathrm{dim}_{\Bbbk} \sHom_A(X,Y) \leq 2,
$$
where $X$ is either a maximal chain module (that is, chain summand of the radical of an indecomposable non-uniserial projective module) or an exceptional band module (Definition \ref{def:exceptional-band}) and $Y$ is a string module; see the discussion following \cite[Proposition~2.2]{A} and \cite[Proposition~2.8]{A}. For certain homogeneous tubes arising from stably biserial algebras, we prove an analogous result. This will serve as one of key lemmas in the proof of Theorem \ref{thm:1.1}.

\begin{lemma}\textnormal{(see Corollary \ref{lem:equal-dimension})} Let $\varLambda$ be a symmetric stably biserial algebra, and let $X$ be a string module lying at the mouth of a homogeneous tube corresponding to a deformed loop in the Brauer graph of $\varLambda$. Then $X$ is a maximal chain module and $\Omega_{\varLambda}(X)$ is an exceptional band module. For any string $\varLambda$-module $Y$, we have
$$\mathrm{dim}_{\Bbbk} \sHom_{\varLambda}(X,Y)=\mathrm{dim}_{\Bbbk} \sHom_{\varLambda}(\Omega_{\varLambda}(X),Y)
\leq 2.$$
\end{lemma}

We note that deformed loops complicate the correspondence between faces and tubes: a face of perimeter $1$ associated with a deformed loop gives rise to two homogeneous tubes interchanged by the syzygy functor, just as a face of perimeter $2$ does for a Brauer graph algebra; see Subsection~\ref{subsec:symmetric-stably-biserial}. The proof of the main theorem therefore requires a separate analysis of the local case, carried out in Proposition~\ref{prop:local-one-face}.

For an algebraically closed field $\Bbbk$, put the local one-face Brauer graph algebra
\[
A_{a,b}=\Bbbk\langle x,y\rangle/(xy,yx,x^a-y^b)
\qquad(a,b\geq2)
\]
and the local symmetric stably biserial algebra
\[
B_m(\lambda,\mu)=\Bbbk\langle u,v\rangle/\big((uv)^m-(vu)^m,\ u^2-\lambda(uv)^m,\ v^2-\mu(uv)^m,\ (uv)^mu,\ (vu)^mv\big),
\]
where $m\geq2$, $\lambda\in \Bbbk^*$ and $\mu\in \Bbbk$. The following is another key lemma in the proof of Theorem~\ref{thm:1.1}.

\begin{lemma}\textnormal{(see Lemma~\ref{lem:local-not-stably-equivalent})}\label{lem:intro-local-not-stably-equivalent}
If $a,b\geq2$, $\max\{a,b\}>2$, $m\geq2$, $\lambda\in \Bbbk^*$ and $\mu\in \Bbbk$, then $A_{a,b}$ and $B_m(\lambda,\mu)$ are not stably equivalent.
\end{lemma}

The proof uses the following properties of the Ziegler closures of their homogeneous tubes:
\begin{itemize}
\item For $A_{a,b}$ with $\max\{a,b\}>2$, the homogeneous tubes whose mouths are maximal chain modules or exceptional band modules form two families, interchanged by the syzygy functor. A tube from one family and a tube from the other have no common generic module in their Ziegler closures; see Lemma~\ref{lem:local-intersection}.
\item For $B_m(\lambda,\mu)$, all tubes in the family $\{\mathcal T(c)\mid c\in \Bbbk\}$ constructed in Lemma~\ref{lem:loc-tubes} have a common generic module in their Ziegler closures. This applies in particular to the distinct tubes $\mathcal T(0)$ and $\mathcal T(\lambda)$, which are interchanged by the syzygy functor; see Lemmas~\ref{lem:loc-syzygy} and~\ref{lem:loc-generic}.
\end{itemize}
The stable Hom bound above forces any hypothetical stable equivalence to send these two tubes over $B_m(\lambda,\mu)$ to tubes from different families over $A_{a,b}$. Krause's homeomorphism in \cite{Krause97} (see Theorem~\ref{thm:Krause-homeomorphism}) preserves the existence of a common generic module in their closures, giving a contradiction. In the application to the local one-face case, the remaining case $a=b=2$ is excluded separately by comparing stable Grothendieck groups.

Stable equivalences of Morita type preserve symmetry by \cite[Corollary 2.4]{L}. Thus Theorem~\ref{thm:1.1} immediately gives the following result. We recall the definition of a stable equivalence of Morita type in Subsection~\ref{subsec:stb-and-center}; see also \cite[Chapter 5]{Z}.

\begin{Thm}\textnormal{(see Theorem \ref{thm:st.M-BGA})}\label{thm:intro}
Let $A$ be a Brauer graph algebra. Then for any basic algebra $B$ with no semisimple summands, if $B$ and $A$ are stably equivalent of Morita type, then $B$ is also a Brauer graph algebra.
\end{Thm}

Since any derived equivalence between self-injective algebras induces a stable equivalence of Morita type, this gives an alternative proof of the result that Brauer graph algebras are closed under derived equivalence, see Corollary \ref{closed-under-derived-equivalence}.

As a byproduct, Theorem \ref{thm:intro} together with the result \cite[Remark 2.10]{GL} gives a solution of the reconstruction problem (see \cite[Introduction]{RR}) posed by Rickard and Rouquier for algebras stably equivalent of Morita type to Brauer graph algebras.

We note that there is another reason on studying stable equivalences of Brauer graph algebras. For Brauer tree algebras, it is well-known that two such algebras are stably equivalent if and only if they are derived equivalent. Furthermore, it has been shown in \cite{Asa2} (see also \cite{LL2}) that any stable equivalence between Brauer tree algebras can be lifted to a derived equivalence. However, for general Brauer graph algebras, this liftability property fails: \cite[Example~6.4]{LL} provides a stable (auto-)equivalence (even of Morita type) that cannot be lifted to a derived equivalence. Therefore, compared with derived equivalences, some aspects of stable equivalences for Brauer graph algebras are still mysterious.

In this paper, our results on stable equivalences are mainly obtained by analyzing the behavior of modules lying at the mouths of tubes in the Auslander-Reiten quiver. In a subsequent paper \cite{CLLX}, using invariants under stable equivalence of Morita type, such as the stable center and maximal tori, we will classify Brauer graph algebras up to stable equivalence of Morita type, and give a short proof of Theorem~\ref{thm:intro} following some idea from \cite{AZ2}.

\medskip
\textbf{Outline.}\; In Section~\ref{sec:prelim}, we recall the required facts about stably biserial algebras, Brauer graph algebras, the stable Grothendieck group and the Ziegler spectrum, followed by the compatibility of stable equivalences with syzygy functors and stable equivalences of Morita type. Section~\ref{app:lem} first gives a stable Hom-dimension formula for self-injective algebras and then applies it to modules associated with deformed loops. Section~\ref{sec:local-case} distinguishes the two local families by their tube closures and then proves the local one-face case. In Section~\ref{sec:stb-closed}, we combine this local argument with the face and module arguments to prove the main results.
	
\medskip
\textbf{Use of AI.}
Sections~\ref{sec:1}--\ref{app:lem} and~\ref{sec:stb-closed} were largely completed independently by the authors, as documented in the earlier arXiv versions of this paper. GPT-6 Astra suggested that the local case of the main theorem could be addressed using Ziegler closures, which led the authors to develop the argument in Section~\ref{sec:local-case}. All mathematical arguments were independently verified and finalized by the authors. The authors take full responsibility for the accuracy, originality, and integrity of the paper.

\section{Preliminaries}\label{sec:prelim}

Throughout we assume that $\Bbbk$ is an algebraically closed field and all algebras considered are finite dimensional $\Bbbk$-algebras. Unless stated otherwise, all modules will be finitely generated left modules. Furthermore, we say that $A$ is a quiver algebra, if $A$ is isomorphic to $\qa$, where $Q$ is a finite quiver and $I$ is an admissible ideal in the path algebra $\Bbbk Q$. We denote by $s(p)$ the source vertex of a path $p$ and by $t(p)$ its terminus vertex. We will write paths from right to left, for example, $p=\alpha_{n}\alpha_{n-1}\cdots\alpha_{1}$ is a path with starting arrow $\alpha_{1}$ and ending arrow $\alpha_{n}$. A path is called a cycle (equally, a closed path) if $s(p)=t(p)$. By abuse of notation we sometimes view an element in $\Bbbk Q$ as an element in the quotient $\qa$ if no confusion can arise.

In this paper, for a finite dimensional $\Bbbk$-algebra $A$ we denote by $\rad(A)$ the Jacobson radical of $A$ and $\soc(A)$ the socle of $A$, that is, the sum of all simple $A$-submodules of $A$. We often write an indecomposable $A$-module $M$ via its Loewy structure, which is represented by a diagram where the $i$-th row corresponds to the simple summands of the completely reducible module $\rad^{i-1}(A)M/\rad^{i}(A)M$. Each number in the diagram denotes a distinct simple module in $A$. For further details, see for example in \cite[Page 174]{Ben}.

Recall (for example from \cite{KZ,Z}) that for an algebra $A$,
	\begin{itemize}
\item the {\it bounded derived category} $\mathcal{D}^b(A)$ is the localization of homotopy category {$\mathcal{K}^b(A)$} by inverting quasi-isomorphisms, {which is a triangulated category with suspension functor the shift functor $[1]$};
\item the {\it stable category} $A\text{-}\underline{\mathrm{mod}}$ is the quotient of the module category $ A\text{-}{\mathrm{mod}}$ modulo the ideal of homomorphisms that factor through projective modules. For $A$-modules $X$ and $Y$, we denote by $\mathrm{Hom}_A(X, Y)$ (resp. $\sHom_A(X, Y)$) the homomorphism space in $A$-mod (resp. $A\text{-}\underline{\mathrm{mod}}$); moreover, for $f\in \mathrm{Hom}_A(X,Y)$, denote by $\underline{f}$ the image of $f$ in $\sHom_A(X, Y)$.
	\end{itemize}
We say that two algebras are {\it derived equivalent} (resp. {\it stably equivalent}) if their derived (resp. stable) categories are equivalent as triangulated categories (resp. as $\Bbbk$-categories).

Regarding the precise definition and properties of the stable equivalence of Morita type, we will give them in Section \ref{subsec:stb-and-center}.

	\subsection{Stably biserial algebras}
		\
	
Recall that a $\Bbbk$-algebra $A$ is called {\it self-injective}  if $_AA$ is an injective $A$-module; and {\it Frobenius} if $_AA\cong DA$ as $A$-modules where $D$ is the usual $\Bbbk$-duality of $A$; and {\it weakly symmetric} if the injective hull of each simple module of $A$ is isomorphic to the projective cover; and {\it symmetric} if $_AA_A\cong DA$ as $A$-$A$-bimodules. Indeed, a basic self-injective algebra is a Frobenius algebra, and symmetric algebras (respectively, Frobenius algebras) are naturally self-injective.

Note that for a self-injective algebra $A$, the syzygy functor $\Omega_{A}$ defines a stable auto-equivalence (even of Morita type) over the stable category $A\text{-}\underline{\mathrm{mod}}$ and $A\text{-}\underline{\mathrm{mod}}$ is a triangulated category (with suspension functor $\Omega_{A}^{-1}$) which is a natural quotient of $\mathcal{D}^b(A)$ (see for example in \cite{KZ,Z}).

	\begin{definition}\textnormal{(\cite[Definition 1]{AZ1})}
Let $Q$ be a quiver and $I$ an admissible ideal of the path algebra $\Bbbk Q$. A self-injective algebra $A$ is said to be stably biserial if $A$ is isomorphic to $\Bbbk Q/I$, where $Q$ and $I$ satisfy the following conditions:
		\begin{enumerate}
\item for each vertex $i\in Q$, the number of outgoing and incoming arrows are less than or equal to $2$;
			
\item for each arrow $\alpha\in Q$, there is at most one arrow $\beta\in Q$ such that $\alpha\beta\notin \alpha\rad(A)\beta+\soc(A)$;
			
\item for each arrow $\alpha\in Q$, there is at most one arrow $\gamma\in Q$ such that $\gamma\alpha\notin \gamma\rad(A)\alpha+\soc(A)$.
		\end{enumerate}
	\end{definition}

For other description of stably biserial algebras we refer to \cite[Proposition A.5]{AIP}.

An algebra $A$ is called {\it special biserial} if it is isomorphic to an algebra of the form $\qa$ where $\Bbbk Q$ is a path algebra and $I$ is an admissible ideal such that the following properties hold:
	
	\begin{enumerate}
\item At every vertex $i$ in $Q$, there are at most two arrows starting at $i$ and there are at most two arrows ending at $i$.
		
\item For every arrow $\alpha$ in $Q$, there exists at most one arrow $\beta$ such that $\beta\alpha\notin I$ and there exists at most one arrow $\gamma$ such that $\alpha\gamma\notin I$.
	\end{enumerate}

A special biserial algebra $A$ is called a {\it string algebra} if the defining ideal $I$ is generated by paths. Note that if $A$ is a stably biserial algebra or a special biserial algebra, then the quotient algebra $\overline{A}$ by factoring the socles of indecomposable projective $A$-modules which are not uniserial is a string algebra, and in this way, the classification of indecomposable $A$-modules is reduced to the classification of indecomposable $\overline{A}$-modules, see Subsection 2.2.2 for further explanation.

Thus, according to the above definitions, self-injective special biserial algebras form a special subclass of stably biserial algebras. In fact, algebras that are stably equivalent to self-injective special biserial algebras have the following characterization.
	
	\begin{theorem}\textnormal{(\cite[Theorem 1]{AZ1})}\label{thm:sta-to-BGA=StB}
Let $A$ be {an indecomposable} self-injective special biserial $\Bbbk$-algebra which is not isomorphic to the Nakayama algebra with $\rad^2(A)=0$. If $B$ is a basic algebra stably equivalent to $A$, then $B$ is stably biserial.
	\end{theorem}

Note that when a stably biserial algebra $A$ is symmetric, there always exists some nice presentation $\Bbbk Q/I$ of $A$, see Theorem \ref{thm:sym-StBA} for the details.

The Auslander-Reiten conjecture for stable equivalences states that if two algebras are stably equivalent, then they have the same number of non-projective simple modules up to isomorphism. In particular, \cite[Theorem 0.1]{P} and \cite[Theorem 2]{AZ1} establish its validity for algebras stably equivalent to special biserial algebras.

	\begin{theorem}\label{thm:AR-conj}
Let $A$ and $B$ be two $\Bbbk$-algebras such that $ A\text{-}\underline{\mathrm{mod}}\cong B\text{-}\underline{\mathrm{mod}}$ and $A$ is special biserial. Then the numbers of isomorphism classes of non-projective simple modules over $A$ and $B$ are the same.
	\end{theorem}

	\subsection{Brauer graph algebras}
	\

	\subsubsection{Basic definitions}\label{subsec:def-BGA}
	\
	
We recall the basic knowledge about Brauer graph algebras defined in \cite{DF}. For the convenience of the following construction, we adopt the notations from \cite{OZ} to define these algebras.
	
	\begin{Def}\textnormal{(\cite[Definition 1.1]{OZ})}\label{ribbon-graph}
A ribbon graph is a tuple $\Gamma=(V,H,s,\iota,\rho)$, where
		\begin{enumerate}
\item $V$ is a finite set whose elements are called vertices;
			
\item $H$ is a finite set whose elements are called half-edges;
			
\item $s: H\rightarrow V$ is a function;
			
\item $\iota: H\rightarrow H$ is an involution without fixed points;
			
\item $\rho: H\rightarrow H$ is a permutation whose cycles correspond to the sets $H_v:=s^{-1}(v)$, $v\in V$.
			
		\end{enumerate}
	\end{Def}
	
Therefore, every ribbon graph defines a graph with vertex set $V$ whose edges are the orbits of $\iota$. An edge $\{h, \iota(h)\}$ is incident to the vertices $s(h)$ and $s(\iota(h))$. In fact, every ribbon graph gives rise to an oriented surface (see for example in \cite[Section 1.1]{OZ}).
	
If $\Gamma$ is a ribbon graph, we write $V(\Gamma)$ and $H(\Gamma)$ for its sets of vertices and half-edges as well as  $H_v(\Gamma)$ for all half-edges $h\in H(\Gamma)$ such that $s(h)=v\in V(\Gamma)$. To simplify the notation, we often write $h^{\pm}:=\rho^{\pm 1}(h)$ for the successor and predecessor of a half-edge $h$ and $\bar{h}$ for its associated edge. The set of all edges is denoted by $E(\Gamma)$. Denote by $\mathrm{val}(v)$ the valency of a vertex $v\in V(\Gamma)$, defined as the number of edges of $\Gamma$ incident to $v$, with a loop counted twice. Unless stated otherwise, we will assume that $\Gamma$ is connected, which means that its underlying graph is connected.
	
		\begin{Def}\textnormal{(\cite[Definition 1.5]{OZ})}\label{BG}
A Brauer graph is a pair $(\Gamma,m)$ consisting of a ribbon graph $\Gamma$ and a function $m: V(\Gamma)\rightarrow \mathbb{Z}_+$.
	\end{Def}
	
The function $m$ in Definition \ref{BG} is referred to as the {\it multiplicity function} and its values as {\it multiplicities}. Frequently, we omit $m$ from the notation and refer to $\Gamma$ as a Brauer graph. A Brauer graph $(\Gamma,m)$ is {\it multiplicity-free} if $m(v)=1$ for every $v\in V(\Gamma)$. In particular,  we call a given vertex $v$ {\it truncated} if $\m(v)\mathrm{val}(v)=1$.

To any Brauer graph $(\Gamma,m)$ one can associate a quiver $Q_\Gamma$ and an admissible ideal of relations $I_\Gamma$ in the path algebra $\Bbbk Q_\Gamma$. If $\Gamma$ is a single edge with two truncated vertices, then we define $Q_\Gamma$ to be a loop $\alpha$ and $I_\Gamma$ to be the ideal generated by $\alpha^2$. Otherwise, we define $Q_\Gamma$ and $I_\Gamma$ as follows.
	
	\begin{enumerate}
\item The vertices of $Q_\Gamma$ correspond to the edges of $\Gamma$ and for every $h\in H(\Gamma)$ with $s(h)$ not truncated, there is an arrow $\alpha_h:\bar{h}\rightarrow\bar{h^+}$. The assignment $\alpha_h\mapsto \alpha_{h^-}$ defines a permutation $\sigma$ of the arrows of $Q_\Gamma$ whose orbits are in bijection with vertices which are not truncated. Hence every arrow $\alpha$ defines a closed path
	$$C_\alpha=\alpha\sigma(\alpha)\cdots\sigma^l(\alpha)$$
where $l+1$ denotes the cardinality of the $\sigma$-orbit of $\alpha$. Every vertex of $Q_\Gamma$ is the starting point of at most two cycles of the form $C_\alpha$. If $\alpha=\alpha_h$, set $\m(C_\alpha):=\m(s(h))$.
	
\item The ideal $I_\Gamma$ is generated by the following set of relations:
	\begin{enumerate}
		\item $$C_\alpha^{\m(C_\alpha)}-C_\beta^{\m(C_\beta)},$$
where $\alpha,\beta\in (Q_\Gamma)_1$ and $t(\alpha)=t(\beta)$, that is $\alpha$ and $\beta$ end at the same edge of $\Gamma$.
		
		\item $$\alpha C_{\sigma(\alpha)}^{\m(C_\alpha)},$$
where $\alpha$ is an arbitrary arrow in $Q_\Gamma$.
		
		\item $$\alpha\beta,$$
where $\alpha,\beta\in (Q_\Gamma)_1$ are composable and $\sigma(\alpha)\neq \beta$;
	\end{enumerate}
	\end{enumerate}
	
The resulting (finite dimensional) $\Bbbk$-algebra $\Bbbk Q_\Gamma/I_\Gamma$ will be denoted by $A_\Gamma$. Note that every nonzero path of $A_\Gamma$ is a subpath of a cycle $C_\alpha^{\m(C_\alpha)}$ for some arrow $\alpha$. The arrow $\alpha_h$ is called an arrow around $s(h)$.

	\begin{Def}\textnormal{(\cite[Definition 1.6]{OZ})}
A $\Bbbk$-algebra $A$ is called a Brauer graph algebra if there exists a Brauer graph $(\Gamma,\m)$ such that $A$ is isomorphic to $A_\Gamma$ as $\Bbbk$-algebras.
	\end{Def}
	
Indeed, we have the following result.
	
	\begin{Thm}\textnormal{(\cite[Theorem 1.1]{Sch})}\label{ssbBGA}
Let $A=\qa$ be a $\Bbbk$-algebra. Then $A$ is a symmetric special biserial algebra if and only if $A$ is a Brauer graph algebra.
	\end{Thm}

Note that since we assume that any Brauer graph $(\Gamma,\m)$ is connected, the resulting Brauer graph algebra $A$ is always indecomposable.

	\subsubsection{Strings and bands}\label{subsec:string-band}
	\

Following \cite{BR,E,D}, we use string and band modules to describe the indecomposable modules over Brauer graph algebras. These correspond to the modules of the first and the second kind in \cite{A}, respectively.

Let $A = \Bbbk Q/I$ be a Brauer graph algebra associated with a Brauer graph $\Gamma$. Recall that for each arrow $\alpha \in Q_1$, denote by $s(\alpha)$ and $t(\alpha)$ the source and target of $\alpha$, respectively. We denote by $\alpha^{-1}$ the formal inverse of $\alpha$, that is, the symbolic arrow satisfying
\[
s(\alpha^{-1}) = t(\alpha), \qquad t(\alpha^{-1}) = s(\alpha).
\]
Let $Q_1^{-1}$ denote the set of all formal inverses of arrows in $Q_1$. We set $(\alpha^{-1})^{-1}=\alpha$. If $w=\alpha_n\cdots\alpha_1$ is a word in $Q_1\cup Q_1^{-1}$, then its inverse is the word $w^{-1}=\alpha_1^{-1}\cdots\alpha_n^{-1}$.

\begin{Def}\label{def:string}
A string is a word $w = \alpha_n \alpha_{n-1} \cdots \alpha_1$ with $\alpha_i \in Q_1 \cup Q_1^{-1}$ such that
\begin{enumerate}
\item $\alpha_{i+1} \neq \alpha_i^{-1}$ for $1\leq i<n$;
\item $s(\alpha_{i+1}) = t(\alpha_i)$ for $1\leq i<n$; and
\item neither $w$ nor $w^{-1}$ contains, as a consecutive subword, a path occurring as a term of one of the defining relations of the ideal $I$ listed in Subsection~\ref{subsec:def-BGA}.
\end{enumerate}
The length of $w$ is the integer $|w| = n$. We say that $w$ is direct (resp. inverse) if each $\alpha_i \in Q_1$ (resp. $\alpha_i \in Q_1^{-1}$). The associated vertex sequence of $w$ is $(x_0,x_1,\ldots,x_n)$, where $x_0=s(\alpha_1)$ and $x_i=t(\alpha_i)$ for $1\leq i\leq n$. We set $s(w)=x_0$ and $t(w)=x_n$.
\end{Def}

Note that a direct string $\alpha_n \cdots \alpha_1$ can also be viewed as a path in the quiver algebra $\Bbbk Q/I$. We shall use these two viewpoints interchangeably.

Next, we include strings of length zero, which correspond to stationary paths and serve as identity elements in the combinatorial setting.

\begin{Def}
A string of length zero is called a zero string. For each vertex $x \in Q_0$, we associate a unique zero string $\varepsilon_x$, corresponding to the idempotent at $x$. We set $|\varepsilon_x|=0$, $s(\varepsilon_x)=t(\varepsilon_x)=x$, and $\varepsilon_x^{-1}=\varepsilon_x$; its vertex sequence is $(x)$. Note that a zero string is considered both direct and inverse.
\end{Def}

If $w=\alpha_n\cdots\alpha_1$ has positive length, a substring of $w$ is a consecutive subword $\alpha_j\cdots\alpha_i$ with $1\leq i\leq j\leq n$. We also regard $\varepsilon_x$ as a substring of a string $w$ whenever $x$ occurs in the vertex sequence of $w$. A string $w$ is called closed if $|w|>0$ and $s(w)=t(w)$.

Strings form the fundamental combinatorial objects used to describe indecomposable modules over special biserial algebras. Among them, some strings exhibit a cyclic behavior, leading to the following notion.

\begin{Def}\label{def:band}
A band in $A$ is a closed string $b$ such that $b^m$ is a string for every integer $m > 0$, and $b$ is not a proper power of another string.
\end{Def}

\begin{Def}\label{def:string-module}
Let $w = \alpha_n \alpha_{n-1} \cdots \alpha_1$ be a string in $A$, where each $\alpha_i \in Q_1 \cup Q_1^{-1}$. Let $(x_0,x_1,\ldots,x_n)$ be the associated vertex sequence of $w$. The string module $M=M(w)$ associated to $w$ is defined as follows.

For each vertex $x \in Q_0$, set
\[
M_x = \bigoplus_{i \, : \, x_i = x} \Bbbk e_i,
\]
that is, $M_x$ is a direct sum of one copy of the field $\Bbbk$ for each occurrence of $x$ in the sequence $x_0, x_1, \dots, x_n$.

For each $\alpha\in Q_1$, define the linear map
$$M_{\alpha} : M_{s(\alpha)}=\bigoplus_{i \, : \, x_i = s(\alpha)} \Bbbk e_i \to M_{t(\alpha)}=\bigoplus_{j \, : \, x_j = t(\alpha)} \Bbbk e_j$$
by
\begin{equation*}
M_{\alpha}(e_i)=\begin{cases}
e_{i+1}, & \text{ if } \alpha_{i+1}=\alpha; \\
e_{i-1}, & \text{ if } \alpha_i=\alpha^{-1}; \\
0, & \text{ otherwise}.
\end{cases}
\end{equation*}

This defines a representation of $Q$ satisfying the relations in $I$, hence an $A$-module.

If $w = \varepsilon_x$ is the zero string at vertex $x$, we define $M(w) := S_x$, the simple $A$-module at $x$.
\end{Def}

Note that each string module is indecomposable. It follows from \cite{BR} that two string modules $M(w_1)$ and $M(w_2)$ are isomorphic if and only if $w_1 = w_2$ or $w_1 = w_2^{-1}$.

Now we recall that there is also the notion of a band module. Each distinct band determines an infinite family of band modules, which can be viewed as indecomposable modules parametrized by a nonzero scalar and a positive integer.

\begin{Def}\label{def:band-module}
Let $b = \beta_n \beta_{n-1} \cdots \beta_1$ be a band in $A$. Fix a positive integer $m \geq 1$ and a nonzero scalar $\lambda \in \Bbbk^*$. Write $x_0 = s(\beta_1)$ and $x_i = t(\beta_i)$ for $1 \leq i < n$.  The band module associated with $(b, m, \lambda)$, denoted by $M(b, m, \lambda)$, is defined as follows.

For each vertex $x \in Q_0$, define
\[
M_x = \bigoplus_{i \, : \, x_i = x} \Bbbk^m e_i,
\]
Thus $M_x$ has one copy of $\Bbbk^m$ for each occurrence of $x$ among $x_0,x_1,\ldots,x_{n-1}$.

For each $1\leq i\leq n$ let $\gamma_i\in Q_1$ denote the underlying arrow of the symbol $\beta_i$ (so $\beta_i=\gamma_i$ or $\beta_i=\gamma_i^{-1}$). We now define the linear map $M_\gamma:M_{s(\gamma)}\to M_{t(\gamma)}$ for each $\gamma\in Q_1$:

\begin{enumerate}
\item For $1\leq i<n$ define:
\begin{itemize}
\item if $\beta_i=\gamma_i$, then
  \[
    M_{\gamma_i}(v e_{i-1}) = v e_i \quad\text{for all }v\in \Bbbk^m,
  \]
and $M_{\gamma_i}$ sends other summands to $0$;
\item if $\beta_i=\gamma_i^{-1}$, then
  \[
    M_{\gamma_i}(v e_i) = v e_{i-1} \quad\text{for all }v\in \Bbbk^m,
  \]
and $M_{\gamma_i}$ sends other summands to $0$.
\end{itemize}
\item For $i=n$, we insert the Jordan block $J_m(\lambda)$:
\begin{itemize}
\item If $\beta_n=\gamma_n$, define
  \[
    M_{\gamma_n}(v e_{n-1}) = J_m(\lambda)\,v e_0 \quad\text{for all }v\in \Bbbk^m,
  \]
and $M_{\gamma_n}$ sends other summands to $0$;
\item If $\beta_n=\gamma_n^{-1}$, define
  \[
    M_{\gamma_n}(v e_0) = J_m(\lambda)^{-1}\,v e_{n-1} \quad\text{for all }v\in \Bbbk^m
  \]
(note that $J_m(\lambda)$ is invertible since $\lambda\neq 0$), and $M_{\gamma_n}$ sends other summands to $0$.
\end{itemize}
\end{enumerate}
All arrow maps $M_\gamma$ not prescribed above are defined to be the zero map. One checks that these maps satisfy the relations in $I$, hence define an $A$-module, which we denote by $M(b,m,\lambda)$.
\end{Def}

Note that each band $b$ gives rise to the family $\{M(b,m,\lambda)\mid m\ge1,\ \lambda\in \Bbbk^*\}$ of indecomposable modules, and that two band modules $M(b,m,\lambda)$ and $M(b',m',\lambda')$ are isomorphic if and only if $m' = m$, $\lambda' =\lambda$, and $b'=b$ or $b'=b^{-1}$ up to taking rotations.

We note that the indecomposable projective $A$-modules which are not uniserial cannot be represented as string or band modules, because the corresponding word contains a path lying in $\soc(A)$, which cannot form a string: it violates condition~(3) of Definition~\ref{def:string}. Equivalently, the string and band $A$-modules are in fact defined over the string algebra $\overline{A}$ which is defined as the quotient algebra of $A$ by factoring the socles of indecomposable projective  $A$-modules which are not uniserial. The same applies to symmetric stably biserial algebras. Indeed, if $A$ is a symmetric stably biserial algebra, then the quotient algebra $\overline{A}$ defined by factoring the socles of indecomposable projective  $A$-modules which are not uniserial is also a string algebra, and therefore all the indecomposable $A$-modules except the indecomposable projective $A$-modules which are not uniserial can also be represented as string or band $A$-modules.

We also note that the notions of string and band depend on the choice of presentation $\Bbbk Q/I$ of the algebra. For instance, in Example~\ref{ex-stably-biserial}, we exhibit different quiver presentations under which a band module may be changed to a string module.

Recall that $\tau_A$ denotes the Auslander--Reiten translation. A connected component of the stable Auslander--Reiten quiver of $A$ is called a \emph{tube of rank $r$} if it is isomorphic, as a translation quiver, to $\mathbb Z A_\infty/\langle\tau^r\rangle$, where $r\geq1$; see \cite[Chapter VIII, Section 4]{ARS}. The vertices of this quiver are indexed by pairs $(i,n)$ with $i\in\mathbb Z/r\mathbb Z$ and $n\geq1$. The vertices $(i,1)$ form the \emph{mouth} of the tube, and the module corresponding to $(i,n)$ has \emph{quasi-length} $n$.

A tube of rank $1$ is called a \emph{homogeneous tube}. It has the form
\[
X_1\rightleftarrows X_2\rightleftarrows X_3
\rightleftarrows\cdots,
\qquad \tau_A X_n\cong X_n \quad(n\geq1).
\]
Here the vertices represent pairwise nonisomorphic indecomposable modules, and the arrows indicate irreducible morphisms in the stable category. The module $X_1$ is the unique module at the mouth. Starting at $X_1$ and moving successively along the displayed chain, the $n$th module $X_n$ has quasi-length $n$. Thus quasi-length counts the layers starting with the mouth as layer $1$. 

Throughout this paper, a homogeneous tube means any tube of rank $1$, including those containing string modules. 

	\subsubsection{Exceptional band modules}\label{subsec:excep-band}
	\

We now turn to a more detailed study of band modules over Brauer graph algebras. The following notion is taken from \cite[Section 2]{A}.

\begin{Def}\label{def:max-string module}
Let $A=\Bbbk Q/I$ be a Brauer graph algebra. A path $p$ from $v$ to $w$ of $Q$ with positive length is called a band line if $p\notin \mathrm{soc}(A)$ and the simple modules $S_v$, $S_w$ of $A$ corresponding to the vertices $v,w$ are not $\Omega_{A}$-periodic. Moreover, $p$ is called a minimal (resp. maximal) band line if it is a minimal (resp. maximal) element in the partial ordered set consisting of all band lines.
\end{Def}

For a Brauer graph $\Gamma$, define $\Gamma^{(1)}$ to be the ribbon graph obtained from $\Gamma$ by deleting all leaves (which correspond to the edges connected with some truncated vertices), and define $\Gamma^{(n)}:=(\Gamma^{(n-1)})^{(1)}$ for every positive integer $n$ inductively. Then there exists some positive integer $r$ such that $\Gamma^{(r)}=\Gamma^{(r+1)}=\cdots$. Denote $\widetilde{\Gamma}=\Gamma^{(r)}$. Moreover, consider the quivers corresponding to these ribbon graphs, it is clear to see that the vertex set of $Q_{\widetilde{\Gamma}}$ is a subset of the vertex set of $Q_{\Gamma}$. For example, $\widetilde{\Gamma_1}=\Gamma_2$ in Example \ref{exa:2-ribbon graph}.

To construct the band lines in a Brauer graph algebra $A$, we use the following lemmas from \cite[Section 2]{A}, which provide criteria for the $\Omega_A$-periodicity of simple modules.

\begin{Lem}\label{lem:paths-in-band=band-line}
Let $X$ be a band module of $A$. Then each simple $A$-module which is a summand of top$(X)$ or soc$(X)$ is not $\Omega_{A}$-periodic.
\end{Lem}

\begin{Lem}
Suppose that $A$ is given by the Brauer graph $(\Gamma,m)$. Then there is a one-to-one correspondence between the isomorphism classes of simple $A$-modules which are not $\Omega_{A}$-periodic and the set of edges of $\widetilde{\Gamma}$.
\end{Lem}

Note that a minimal band line $p$ in $A$ determines an arrow $\alpha_p$ with $s(\alpha_p)=s(p)$ and $t(\alpha_p)=t(p)$ in $Q_{\widetilde{\Gamma}}$, and this correspondence is a bijection.

The following lemma provides us a method to construct a band from a string whose ending vertices corresponding to simple modules which are not $\Omega$-periodic.

\begin{Lem}\label{lem-string-to-band} {\rm(\cite[Lemma 2.5]{A})}
Let $A$ be a Brauer graph algebra, and let
\[
w = p_l p_{l-1} \cdots p_1
\]
be a string in $A$, where each $p_i$ is either a direct string or an inverse string, appearing alternately, and both $s(w)$ and $t(w)$ correspond to simple $A$-modules that are not $\Omega_A$-periodic. Then there exists a band
\[
b = p_m p_{m-1} \cdots p_l \cdots p_1
\]
in $A$, where again each $p_i$ is either a direct string or an inverse string, appearing alternately.
\end{Lem}

Now we define the exceptional band module in \cite{A}.

\begin{Def}\label{def:exceptional-band}
A band $b$ of $A$ is called exceptional if it can be expressed up to taking rotations and inverses in the form $$b = p_{2k} p_{2k-1}^{-1} \cdots p_{2}p_{1}^{-1},$$ where each $p_1, p_3, \ldots, p_{2k-1}$ is a minimal band line and each $p_2, p_4, \ldots, p_{2k}$ is a maximal band line.

A band module $X$ of $A$ is called exceptional if it is defined by an exceptional band $b$, with the vector space at each vertex of $b$ being the one-dimensional space $\Bbbk$, and all but one of the arrows act as identity maps, while one arrow acts as multiplication by some scalar $\lambda\in \Bbbk^{*}$, that is, $X{\cong} M(b,1,\lambda)$.

If $b$ is a non-exceptional band, then we call a band module of the form $M(b,1,\lambda)$ a non-exceptional band module.
\end{Def}

The construction of the following string module is given in \cite[Proposition 2.3]{A}. We provide a concrete proof to show that the string module constructed in \cite[Proposition 2.3]{A} in fact does not lie in a tube of rank $1$ of the Auslander-Reiten quiver.

\begin{proposition}\label{prop:non-exceptional-tube}
Let $A = \Bbbk Q/I$ be a Brauer graph algebra, and let $M$ be a non-exceptional band module. For any $n \in \mathbb{Z}_+$, there exists a string module $Y$ that does not lie in a tube of rank $1$ of the Auslander-Reiten quiver of $A$, such that
\[
\mathrm{dim}_{\Bbbk} \sHom_A(M, Y) > n.
\]
\end{proposition}

\begin{proof}
Suppose $M$ is given by the band
\[
p_{2m} p_{2m-1}^{-1} \cdots p_{2} p_{1}^{-1}
\]
of $A$. Then by Lemma \ref{lem:paths-in-band=band-line}, each $p_i$ is a band line. Since $M$ is not an exceptional band module, without loss of generality, at least one of the following possibilities occurs:

\begin{itemize}
\item[(1)] $p_1$ and $p_{2k}$ are not minimal band lines.
\item[(2)] $p_1$ and $p_{2k}$ are not maximal band lines.
\item[(3)] $p_1$ is not a minimal band line, and $p_{2k+1}$ is not a maximal band line (they may be the same line).
\end{itemize}

The third case, in contrast to the first and second, arises for band modules in which all even band lines are either simultaneously maximal or simultaneously minimal. Note that this proposition is invariant under the $\Omega_A$-shift, and since the maximal band lines are mapped to minimal band lines under the $\Omega_A$-shift, we do not need to consider case (2).

Consider case (1). Denote by $q_1$ (resp. $q_2$) the minimal band line that is contained in $p_1$ (resp. $p_{2k}$) and has a common starting arrow with $p_1$ (resp. $p_{2k}$). Then there exists an epimorphism $f : M \to Z$, where $Z$ is a string module which is given by $$w_1=q_2p_{2k-1}^{-1}\cdots p_{2}q_1^{-1}.$$

Since the simple modules corresponding to $s(w_1)$ and $t(w_1)$ are not $\Omega_A$-periodic, it follows from Lemma \ref{lem-string-to-band} that there exists a closed string $w = w_2w_1$ of $A$, where $w_2$ starts with a formal inverse in $Q_1^{-1}$ and ends with an arrow in $Q_1$. For any natural number $n$, consider the string module $Y$ corresponding to the string
$$w_1w^n$$
It is easy to see that a monomorphism $h : Z \to Y$ and a morphism $hf : M \to Y$ correspond to any occurrence of the substring $Z$ in $Y$. Since $hf$ induces a nonzero map $\operatorname{top}(M) \to \operatorname{top}(Y)$ (the image of $p_i$ is nonzero), it is nonzero in the stable category.

We now show that such a module $Y$ does not lie in a tube of rank $1$ of the stable Auslander-Reiten quiver of $A$. It suffices to prove that $\tau(Y)\not\cong Y$, where $\tau$ is the Auslander-Reiten translation of $A$. First, note that $q_1$ and $q_2$ are not maximal direct strings, since there exist paths $p_1$ and $p_{2k}$ containing them that are not in the socle of $A$. Hence, $\tau(Y)$ is a string module obtained by attaching a cohook to each side of $w_1 w^n$ (see for example in \cite{BR}), which is given by
\[
p''^{-1}\beta w_1 w^n \alpha^{-1}p',
\]
where $p'$ and $p''$ are maximal direct strings and $\alpha, \beta \in Q_1$. Therefore, $\tau(Y)\not\cong Y$.

Now consider case (3).

Let $k \neq 0$. Let $q_1$ denote the minimal band line sharing a common starting arrow with $p_1$, and let $q_2$ denote the maximal band line sharing a common ending arrow with $p_{2k+1}$. Then there exists a morphism from $M$ to an arbitrary string module that contains the substring
\[
w_1 = q_2^{-1}p_{2k}\cdots p_2 q_1^{-1}.
\]
As in the previous case, there exists a string module $Y$ of the form
\[
w_1w^n,
\]
where $w$ is a closed string containing $w_1$, such that
\[
\mathrm{dim}_{\Bbbk} \sHom_A(M, Y) > n.
\]
Note that $q_1$ is not a maximal direct string, since there exists a path $p_1$ containing it that are not in the socle of $A$. Thus, $\tau(Y)$ is a string module obtained by attaching a cohook to each side of $w_1w^n$, which will not isomorphic to $Y$.

Let $k=0$, and choose a path $q$ such that $p_1q\in\soc(A)$. Note that $q$ is also a band line, and since $p_1$ is neither a maximal nor a minimal band line, $q$ is also neither a maximal nor a minimal band line. It follows that we have the following decompositions
\[
p_1 = p' p'', \quad q = q' q'',
\]
where $p', p'', q', q''$ are band lines. Then $p''q'$ is a band line. For any string module $Z$ containing $p''q'$ as a substring, there exists a morphism $f : M \to Z$. We claim that $f$ is nonzero in the stable category. Indeed, suppose to the contrary that $f$ is zero in the stable category. Then there exists a morphism $g : I(M) \to Z$ such that $f = g i$, where $I(M)$ denotes the injective hull of $M$ and $i : M \to I(M)$ is the canonical embedding. However, such a morphism $g$ cannot simultaneously map the top of $I(M)$ (corresponding to the simple module associated with $s(p_1q)$) to a composition factor corresponding to a vertex in $q'$, and at the same time map the simple module corresponding to $s(p_1)$ to that corresponding to $s(p'')$. This yields a contradiction.

As in the previous case, there exists a string module $Y$ of the form
\[
(p''q')^{-1}w^n,
\]
where $w$ is a closed string containing $(p''q')^{-1}$, such that
\[
\mathrm{dim}_{\Bbbk} \sHom_A(M, Y) > n.
\]
Note that $p''q'$ is not a maximal direct string, since there exists a path $p'p''q'$ containing it that are not in the socle of $A$. Thus, $\tau(Y)$ is a string module obtained by attaching a cohook to each side of $(p''q')^{-1}w^n$, which will not isomorphic to $Y$.
\end{proof}

\subsubsection{Faces and double-faces}
\

We now recall the connection between exceptional band modules (which corresponds to exceptional modules of second kind  in \cite{A}) in a Brauer graph algebra $A$ and double-faces in its Brauer graph $\Gamma_A$. We begin by recalling the definitions of faces and double-faces in a Brauer graph in \cite{OZ}, which correspond to Green walks in \cite{GreenJA} (or $G$-cycles in \cite{A}) and double-stepped Green walks in \cite{D} (or $G^2$-cycles in \cite{A}), respectively.

\begin{Def}
Let $\Gamma$ be a ribbon graph. A face of perimeter $p$ of $\Gamma$ is an equivalence class of primitive cyclic sequences $F=(h_1,h_2,\cdots,h_p)$, where $h_1,\cdots,h_p\in H(\Gamma)$ such that $h_{i+1}=\iota\rho(h_{i})$ for all $i$. We call a sequence $F$ primitive, if $F$ is not a power of another sequence of this form. Two sequences of such kind are considered equivalent if they agree after a cyclic permutation of their entries.
\end{Def}

\begin{Def}
Let $\Gamma$ be a ribbon graph. A double-face of perimeter $p$ of $\Gamma$ is an equivalence class of primitive cyclic sequences $F=(h_1,h_2,\cdots,h_p)$, where $h_1,\cdots,h_p\in H(\Gamma)$ such that $h_{i+1}=(\iota\rho)^{2}(h_{i})$ for all $i$. Two sequences of such kind are considered equivalent if they agree after a cyclic permutation of their entries.
\end{Def}

So, in particular, the faces give a partition on the set $H(\Gamma)$ of half-edges and the double-faces also give a partition on $H(\Gamma)$. Note that both notions are free of the multiplicity $m$. We give the following example to help readers understand these concepts.

\begin{example}\label{exa:2-ribbon graph}
Let $\Gamma_1$ and $\Gamma_2$ be the Brauer graphs with multiplicity one and the ribbon graphs are given as follows (the orientation around each vertex is clockwise), respectively.
	\begin{center}
\tikzset{every picture/.style={line width=0.75pt}} %set default line width to 0.75pt

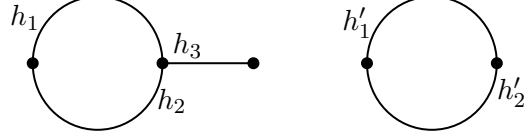
\begin{tikzpicture}[x=0.75pt,y=0.75pt,yscale=-1,xscale=1]
%uncomment if require: \path (0,235); %set diagram left start at 0, and has height of 235

%Shape: Ellipse [id:dp5351628263682573]
\draw   (61.1,94.92) .. controls (61.1,76.46) and (75.68,61.5) .. (93.67,61.5) .. controls (111.67,61.5) and (126.25,76.46) .. (126.25,94.92) .. controls (126.25,113.37) and (111.67,128.33) .. (93.67,128.33) .. controls (75.68,128.33) and (61.1,113.37) .. (61.1,94.92) -- cycle ;
%Shape: Ellipse [id:dp22668671113333483]
\draw   (228.86,94.92) .. controls (228.86,76.46) and (243.45,61.5) .. (261.44,61.5) .. controls (279.43,61.5) and (294.02,76.46) .. (294.02,94.92) .. controls (294.02,113.37) and (279.43,128.33) .. (261.44,128.33) .. controls (243.45,128.33) and (228.86,113.37) .. (228.86,94.92) -- cycle ;
%Straight Lines [id:da3132565897170516]
\draw    (126.25,94.92) -- (171.86,94.92) ;
%Shape: Circle [id:dp7255337491938683]
\draw  [fill={rgb, 255:red, 0; green, 0; blue, 0 }  ,fill opacity=1 ] (169.36,94.92) .. controls (169.36,93.54) and (170.48,92.42) .. (171.86,92.42) .. controls (173.24,92.42) and (174.36,93.54) .. (174.36,94.92) .. controls (174.36,96.3) and (173.24,97.42) .. (171.86,97.42) .. controls (170.48,97.42) and (169.36,96.3) .. (169.36,94.92) -- cycle ;
%Shape: Circle [id:dp9224289625323931]
\draw  [fill={rgb, 255:red, 0; green, 0; blue, 0 }  ,fill opacity=1 ] (123.75,94.92) .. controls (123.75,93.54) and (124.87,92.42) .. (126.25,92.42) .. controls (127.63,92.42) and (128.75,93.54) .. (128.75,94.92) .. controls (128.75,96.3) and (127.63,97.42) .. (126.25,97.42) .. controls (124.87,97.42) and (123.75,96.3) .. (123.75,94.92) -- cycle ;
%Shape: Circle [id:dp8545479082638043]
\draw  [fill={rgb, 255:red, 0; green, 0; blue, 0 }  ,fill opacity=1 ] (58.6,94.92) .. controls (58.6,93.54) and (59.72,92.42) .. (61.1,92.42) .. controls (62.48,92.42) and (63.6,93.54) .. (63.6,94.92) .. controls (63.6,96.3) and (62.48,97.42) .. (61.1,97.42) .. controls (59.72,97.42) and (58.6,96.3) .. (58.6,94.92) -- cycle ;
%Shape: Circle [id:dp29041330982278835]
\draw  [fill={rgb, 255:red, 0; green, 0; blue, 0 }  ,fill opacity=1 ] (226.36,94.92) .. controls (226.36,93.54) and (227.48,92.42) .. (228.86,92.42) .. controls (230.24,92.42) and (231.36,93.54) .. (231.36,94.92) .. controls (231.36,96.3) and (230.24,97.42) .. (228.86,97.42) .. controls (227.48,97.42) and (226.36,96.3) .. (226.36,94.92) -- cycle ;
%Shape: Circle [id:dp4212398147160299]
\draw  [fill={rgb, 255:red, 0; green, 0; blue, 0 }  ,fill opacity=1 ] (291.52,94.92) .. controls (291.52,93.54) and (292.63,92.42) .. (294.02,92.42) .. controls (295.4,92.42) and (296.52,93.54) .. (296.52,94.92) .. controls (296.52,96.3) and (295.4,97.42) .. (294.02,97.42) .. controls (292.63,97.42) and (291.52,96.3) .. (291.52,94.92) -- cycle ;

% Text Node
\draw (48.34,64.2) node [anchor=north west][inner sep=0.75pt]   [align=left] {$h_1$};
% Text Node
\draw (130,78.23) node [anchor=north west][inner sep=0.75pt]   [align=left] {$h_3$};
% Text Node
\draw (122,106.3) node [anchor=north west][inner sep=0.75pt]   [align=left] {$h_2$};
% Text Node
\draw (215,65.53) node [anchor=north west][inner sep=0.75pt]   [align=left] {$h_1'$};
% Text Node
\draw (292.66,100.29) node [anchor=north west][inner sep=0.75pt]   [align=left] {$h_2'$};
\end{tikzpicture}
	\end{center}

The graph $\Gamma_1$ contains two faces: one of perimeter $2$, $(h_1,h_2)$, and another one of perimeter $4$, $(\iota(h_1),\iota(h_3),h_3,\iota(h_2))$. From these, we obtain four double-faces: $(h_1)$, $(h_2)$, $(\iota(h_1),h_3)$, and $(\iota(h_3),\iota(h_2))$.

In contrast, $\Gamma_2$ has two faces, both of perimeter 2: $(h_1',h_2')$ and $(\iota(h_1'),\iota(h_2'))$. From these, we obtain four double-faces: $(h_1')$, $(h_2')$, $(\iota(h_1'))$, and $(\iota(h_2'))$.
\end{example}

As the example above illustrates, a face of even perimeter gives rise to two distinct double-faces. It is worth noting that a face with an odd perimeter, however, contributes to exactly one.

The following proposition shows that every double-face $F = (h_1, h_2, \ldots, h_k)$ in the graph $\widetilde{\Gamma}$ determines an exceptional band $b = p_{2k}^{-1} p_{2k-1} \cdots p_{2}^{-1} p_{1}$ of $A$ such that  $p_1, p_3, \ldots, p_{2k-1}$ are minimal band lines in $A$ corresponding to the arrows $\alpha_{h_1}, \alpha_{h_2}, \ldots, \alpha_{h_k}$ in the quiver $Q_{\widetilde{\Gamma}}$.

\begin{Prop}\textnormal{(\cite[Lemma 2.7]{A})}\label{prop:double-faces and exceptional bands}
Let $A$ be a Brauer graph algebra with associated Brauer graph $\Gamma$. Let $M$ be an exceptional band $A$-module defined by an exceptional band
\[
b = p_{2k}^{-1} p_{2k-1} \cdots p_{2}^{-1} p_{1},
\]
where $p_1, p_3, \ldots, p_{2k-1}$ are minimal band lines. Suppose that for each $1 \leq i \leq k$, the path $p_{2i-1}$ goes from vertex $v_i$ to vertex $w_i$ in the quiver $Q_\Gamma$.

Then the following statements hold.

\begin{enumerate}
\item There exists a double-face $F = (h_1, h_2, \ldots, h_k)$ in $\widetilde{\Gamma}$ such that for each $i$, the arrow $\alpha_{h_i}$ in the quiver $Q_{\widetilde{\Gamma}}$ is determined by the minimal band line $p_{2i-1}$ in $A$.

\item The syzygy $\Omega_A(M)$ is also an exceptional band module. After applying a suitable cyclic permutation to the band defining $\Omega_A(M)$, $\Omega_A(M)$ can be defined by an exceptional band
    \[
	b' = q_{2k}^{-1} q_{2k-1} \cdots q_{2}^{-1} q_{1},
    \]
where for each $1 \leq i \leq k$, $q_{2i-1}$ is a minimal band line goes from vertex $w_i$ to vertex $v_{i+1}$ in the quiver $Q_\Gamma$. Furthermore, the minimal band lines $q_1, q_3, \ldots, q_{2k-1}$ correspond to the double-face $F' = (\iota\rho(h_1), \iota\rho(h_2), \ldots, \iota\rho(h_k))$ in $\widetilde{\Gamma}$, that is, the arrow $\alpha_{\iota\rho(h_i)}$ in the quiver $Q_{\widetilde{\Gamma}}$ is determined by the minimal band line $q_{2i-1}$ in $A$ for each $i$.

\item If $\iota\rho(h_1)=h_i$ for some $i=1,2,\cdots,k$, then $F=F'$ up to a rotation. Otherwise, $G:= (h_1, \iota\rho(h_1), \ldots, h_{k},\iota\rho(h_k))$ is a face in $\widetilde{\Gamma}$.
\end{enumerate}
\end{Prop}

We end this subsection by the following example.

\begin{example}\label{ex:two-exceptional-band} (Example \ref{exa:2-ribbon graph} revisited)
Denote the vertices in $Q_{\Gamma_1}$ corresponding to $\bar{h_1}, \bar{h_2},\bar{h_3}$ by $v_1,v_2,v_3$ and the vertices in $Q_{\Gamma_2}$ corresponding to $\bar{h_1'}, \bar{h_2'}$ by $v_1,v_2$. Then the quivers associated to the ribbon graphs $\Gamma_1$ and $\Gamma_2$ are given as follows.
	
	$$% https://tikzcd.yichuanshen.de/#N4Igdg9gJgpgziAXAbVABwnAlgFyxMJZABgBoBGAXVJADcBDAGwFcYkQBFAfWAB1eA4vQC2w+l3IBfRCEml0mXPkIoArBWp0mrdtz6CRYrgCZps+SAzY8BIuQ00GLNjPLmF15UQDMDrc-Zjd0tFGxVkY1JiTScdGW9gqyVbFAA2P1iXEDc5D2TwgA4M7SygyU0YKABzeCJQADMAJwhhJEiQHAgkXxA4AAssepw2mgAjGDAoJDJ-OJB+JjQ+8WA+iUkQGkZ6ccYABVCvGSwwbFhgppakdQ6uxHTegaHuxxL2BcYllbXyAHINrY7GD7Q4pEAnM5sXIgS6te40TrXGj9QbDRAAWgemXevEWyx4-HwOHoAAo1sZfgBKAEgba7A6eMGMGDPaGwl63EaPVFIdE3bEyD5fAm8Imk8nUza0oEgxkqWks4Zs5pw9qIxAAFleAUFuM++P0YrJEklyqumoRdx6AvmeuFqy43g2ZrhPXV7RRz0Q9ml9NB8saWCqfWG2rmQoN5OdFnZ8M5iBunrRPTpwIZ+XYEKw5zDWQj3xM-wuKqR8YeSd5D1TsozMmZz1zOLxK0JEGJxr+psokiAA
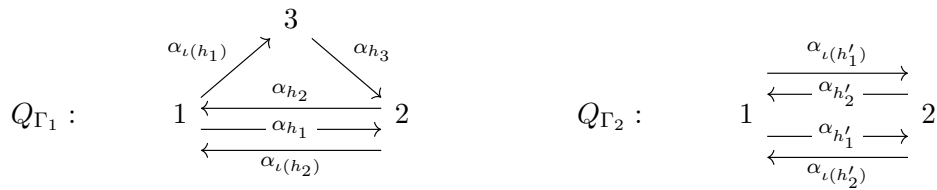
\begin{tikzcd}
              &                                                                                           & 3 \arrow[rd, "\alpha_{h_3}"] &                                                                                            &  &               &                                                                                                           &  &                                                                                                           \\
Q_{\Gamma_1}: & 1 \arrow[rr, "\alpha_{h_1}" description, shift right=2] \arrow[ru, "\alpha_{\iota(h_1)}"] &                              & 2 \arrow[ll, "\alpha_{\iota(h_2)}", shift left=5] \arrow[ll, "\alpha_{h_2}"', shift right] &  & Q_{\Gamma_2}: & 1 \arrow[rr, "\alpha_{h_1'}" description, shift right=3] \arrow[rr, "\alpha_{\iota(h_1')}", shift left=6] &  & 2 \arrow[ll, "\alpha_{\iota(h_2')}", shift left=6] \arrow[ll, "\alpha_{h_2'}" description, shift right=3]
\end{tikzcd}$$
	
Let $M$ be an exceptional band $A_{\Gamma_1}$-module defined by the exceptional band $(\alpha_{h_3}\alpha_{\iota(h_1)})^{-1}\alpha_{h_1}$, which can be represented by
	
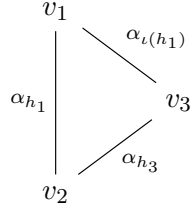
\begin{center}
\begin{tikzcd}
v_1 \arrow[rd, "\alpha_{\iota(h_1)}", no head] \arrow[dd, "\alpha_{h_1}"', no head] &                                         \\
                                                                                    & v_3 \arrow[ld, "\alpha_{h_3}", no head] \\
v_2                                                                                 &
\end{tikzcd}
	\end{center}
Therefore, $\alpha_{h_1}$ is a minimal band line in $M$, which correspond to the double-face $\{h_1'\}$ in $\widetilde{\Gamma_1}=\Gamma_2$, and $\alpha_{h_3}\alpha_{\iota(h_1)}$ is a maximal band line of $A_{\Gamma_1}$. Note that $\alpha_{h_3}\alpha_{\iota(h_1)}$ is also a minimal band line corresponding to the double-face $\{\iota(h_1')\}$.

At the same time, $\Omega_{A_{\Gamma_1}}(M)$ can be defined by the exceptional band $\alpha_{\iota(h_2)}^{-1}\alpha_{h_2}$, which can be represented by
	
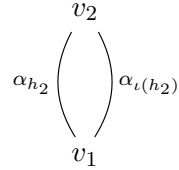
\begin{center}
\begin{tikzcd}
v_2 \arrow[dd, "\alpha_{h_2}"', no head, bend right] \arrow[dd, "\alpha_{\iota(h_2)}", no head, bend left] \\
                                                                                                           \\
v_1
\end{tikzcd}
	\end{center}
Therefore, $\alpha_{h_2}$ is a minimal band line in $\Omega_{A_{\Gamma_1}}(M)$, which correspond to the double-face $\{h_2'\}$ in $\widetilde{\Gamma_1}=\Gamma_2$. Moreover, the face $\{h_1',h_2'\}$ corresponds to the minimal band lines of $M$ and $\Omega_{A_{\Gamma_1}}(M)$.
\end{example}

\subsection{Symmetric stably biserial algebras}\label{subsec:symmetric-stably-biserial}
\

Recall that we have introduced the stably biserial algebras in Section 2.1. We now concentrate on the symmetric stably biserial algebras. Note that every symmetric stably biserial algebra can be represented by a Brauer graph $(\Gamma, m)$ together with a special set $\mathcal{L}$ of loops in $Q_\Gamma$. This result was proved in \cite[Section 5]{AZ1}, and a clearer exposition can be found in \cite[Theorem 2.4]{AZ2}.

\begin{Thm}\label{thm:sym-StBA}
Any symmetric stably biserial algebra $A$ has a presentation $\Bbbk Q_\Gamma/I$, where $Q_\Gamma$ is a quiver associated with a Brauer graph $(\Gamma,m)$, and the ideal of relations $I$ is generated by
	\begin{enumerate}
\item $\alpha\beta=0$ for all $\alpha,\beta\in (Q_\Gamma)_1$, $\beta\neq \sigma(\alpha)$ and $\alpha\notin\mathcal{L}$;
\item $C_\alpha^{m(C_\alpha)}=C_\beta^{m(C_\beta)}$ for all $\alpha,\beta\in (Q_\Gamma)_1$ with $t(\alpha)=t(\beta)$;
\item $\alpha_i^2=t_{\alpha_i} C_{\alpha_i}^{m(C_{\alpha_i})}$ for each $\alpha_i\in\mathcal{L}$;
\item $C_\alpha^{m(C_\alpha)}\beta=0$ for all $\alpha,\beta\in (Q_\Gamma)_1$,
	\end{enumerate}
where $\mathcal{L}=\{\alpha_1,\cdots,\alpha_n\}\subseteq (Q_\Gamma)_1$ such that each $\alpha_i$ is a loop with $\sigma(\alpha_i)\neq \alpha_i$, and $t_{\alpha_i}\in \Bbbk^*$.

Moreover, when $\mathrm{char}(\Bbbk)\neq 2$, any symmetric stably biserial algebra is isomorphic to an algebra in this form with $\mathcal{L}=\varnothing$, in which case it is a Brauer graph algebra.
\end{Thm}

The loops in $\mathcal L$ are called \emph{deformed loops}. Compare to the defining relations of a Brauer graph algebra, the above relations $(3)$ are new. Observe that the fact that $Q_\Gamma$ contains deformed loops implies that $\Gamma$ contains loops. Note also that the associated string algebra of the above symmetric stably biserial algebra $A=\Bbbk Q/I$ and the associated string algebra of the Brauer graph algebra defined by $(\Gamma,m)$ are the same (cf. Subsection 2.2.2).

Throughout the paper, presentations of symmetric stably biserial algebras are taken in the form described in Theorem~\ref{thm:sym-StBA}.

Here we provide an example of a symmetric stably biserial algebra which is not isomorphic to a special biserial algebra.

\begin{Ex1}\label{ex-stably-biserial} {\rm(cf. \cite[Example A.7]{AIP})}
Let $\Bbbk$ be a field of characteristic $2$ and let $A$ be a $\Bbbk$-algebra given by the quiver
$$\xymatrix{
1\ar@/^/[r]^{\alpha}\ar@(dl,ul)^{\gamma} & 2\ar@/^/[l]^{\beta} \\
 }$$
with relations $$\gamma\beta\alpha=\beta\alpha\gamma=\gamma^2, \quad \gamma^3=\alpha\gamma^2=\beta\alpha\gamma\beta=\alpha\beta=0.$$
Then $A$ is a symmetric stably biserial algebra (with linear form $\varphi: A\to \Bbbk$ by $\varphi(\gamma\beta\alpha)=1$, $\varphi(\alpha\gamma\beta)=1$ and $\varphi(p)=0$ for each path $p\notin \mathrm{Soc}A$) but not isomorphic to any symmetric special biserial algebra. Note that $A$ is given by the multiplicity-free Brauer graph
$$
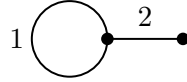
\begin{tikzpicture}
\draw (0,0) circle (0.5);
\fill (0.5,0) circle (0.5ex);
\node at(-0.7,0){$1$};
\node at(1.0,0.3){$2$};
\draw (0.5,0)--(1.5,0);
\fill (1.5,0) circle (0.5ex);
\end{tikzpicture}$$
with $\mathcal{L}=\{\gamma\}$ and $t_{\gamma}=1$. We have the following structure of the indecomposable projective $A$-modules: \xymatrix@R=0.5pc@C=0.8pc {&&1\ar@{-}[dl]\ar@{-}[dr]&& \\ P_1= & 1 \ar@{-}[d]\ar@{--}[ddr]& & 2\ar@{-}[d] & \\ & 2 \ar@{-}[dr] &  & 1 \ar@{-}[dl] & \\ && 1 & ,& } \xymatrix@R=0.5pc@C=0.8pc {&2\ar@{-}[d]& \\ P_2= & 1\ar@{-}[d] & \\ & 1 \ar@{-}[d] & \\ & 2 & .}

Moreover, let $B$ be the Brauer graph algebra whose Brauer graph is the same as the one of $A$. Then $B$ is given by the quiver
$$\xymatrix{
1\ar@/^/[r]^{\alpha}\ar@(dl,ul)^{\gamma} & 2\ar@/^/[l]^{\beta} \\
 }$$
with relations $$\gamma\beta\alpha=\beta\alpha\gamma, \quad \gamma^2=\beta\alpha\gamma\beta=\alpha\beta=0.$$
We have the following structure of the indecomposable projective $B$-modules: \\ \xymatrix@R=0.5pc@C=0.8pc {&&1\ar@{-}[dl]\ar@{-}[dr]&& \\ Q_1= & 1 \ar@{-}[d]& & 2\ar@{-}[d] & \\ & 2 \ar@{-}[dr] &  & 1 \ar@{-}[dl] & \\ && 1 & ,& } \xymatrix@R=0.5pc@C=0.8pc {&2\ar@{-}[d]& \\ Q_2= & 1\ar@{-}[d] & \\ & 1 \ar@{-}[d] & \\ & 2 & .}

Suppose that $A$ is isomorphic to some Brauer graph algebra $C$. Then the associated Brauer graphs of $A$ and $C$ are the same (see \cite[Lemma 3.1]{AZ2}). It follows that the Brauer graph algebras $B$ and $C$ are isomorphic. So, in order to prove that $A$ is not isomorphic to any Brauer graph algebra, it is enough to show that $A$ is not isomorphic to $B$. In fact, $A$ and $B$ are not stably equivalent. The reason is as follows. Each indecomposable $A$-module (resp. $B$-module) at the mouth of a tube of rank $1$ of the Auslander-Reiten quiver of $A$ (resp. $B$) is of the form
\[\begin{tikzcd}[row sep=7, column sep=0]
            & {1} & \\
            {X_t =} & {2} & t \\
            & {1} &
            \arrow[shift left=2, bend left=15, dashed, no head, from=1-2, to=3-2]
            \arrow[no head, from=1-2, to=2-2]
            \arrow[no head, from=2-2, to=3-2]
        \end{tikzcd}\]
with $t\in \Bbbk$, where the dashed line means that \(\gamma-t\beta\alpha\) acts as zero on the elements of $\top(X_t)$. A calculation shows that $\Omega_{A}(X_t)\cong X_{1+t}$ and $\Omega_{B}(X_t)\cong X_{t}$ for each $t\in \Bbbk$. If $F:B\text{-}\mathrm{\underline{mod}}\rightarrow A\text{-}\mathrm{\underline{mod}}$ is a stable equivalence, then Proposition~\ref{prop:stable-equivalence-self-injective} implies that $F(X_0)$ is isomorphic to some $X_t$ with $t\in \Bbbk$. By Proposition~\ref{prop:stable-equivalence-syzygy}, we have
$$X_{t+1}\cong\Omega_{A}(X_t)\cong\Omega_{A}F(X_0)\cong F\Omega_{B}(X_0)\cong F(X_0)\cong X_t,$$
a contradiction.

However, when $\mathrm{char}(\Bbbk)\neq 2$, the change of basis given by
\[
\alpha'=\alpha,\quad \beta'=\beta,\quad \gamma'=\gamma-\frac{\beta\alpha}{2}
\]
yields an isomorphism $A\cong B$. Under this isomorphism, certain band modules over $A$ are sent to string modules over $B$. For instance, the band module $X=A(\gamma-\frac{\beta\alpha}{2})$ over $A$ is mapped to the string module $Y=B\gamma$ over $B$, with the following Loewy structures:
\[
\begin{tikzcd}[row sep=7, column sep=0]
            & {1} & \\
            {X =} & {2} & \frac{1}{2}, \\
            & {1} &
            \arrow[shift left=2, bend left=15, dashed, no head, from=1-2, to=3-2]
            \arrow[no head, from=1-2, to=2-2]
            \arrow[no head, from=2-2, to=3-2]
        \end{tikzcd}
\qquad
\begin{tikzcd}[row sep=7, column sep=0]
            & {1}  \\
            {Y =} & {2} \\
            & {1}
            \arrow[no head, from=1-2, to=2-2]
            \arrow[no head, from=2-2, to=3-2]
        \end{tikzcd}.
\]
\end{Ex1}

It is important to note that although every symmetric stably biserial algebra admits a presentation of the form $\Bbbk Q/I$ described in Theorem~\ref{thm:sym-StBA}, an algebra of this form need not be symmetric, as follows from \cite[Remark~3.2]{AZ2}. Moreover, even in characteristic $2$, there exist stably biserial algebras of the form $\Bbbk Q/I$ given in Theorem~\ref{thm:sym-StBA} that are isomorphic to special biserial algebras, as mentioned in the final paragraph of the proof of \cite[Lemma~3.1]{AZ2}. The following lemma makes these observations explicit.

\begin{lemma}\label{lem:stb-local-iso}
Let $0\neq\lambda_1\in \Bbbk$, and consider the stably biserial algebra
$$B=\Bbbk\langle x,y\rangle/\langle xy-yx,\,x^2-\lambda_1yx,\,y^2-\lambda_2yx,\,x^2y,\,y^2x\rangle.$$
The set $\mathcal{L}$ of deformed loops is given by
$$\mathcal{L}=\begin{cases}
	\{x\}, & \text{if }\lambda_2=0,\\
	\{x,y\}, & \text{if }\lambda_2\neq0.
\end{cases}$$
Then $B$ is symmetric if and only if $\lambda_1\lambda_2\neq1$. Moreover, if $B$ is symmetric, then $B$ is isomorphic to the special biserial algebra
$$A=\Bbbk\langle x,y\rangle/\langle x^2-y^2,\,xy,\,yx\rangle.$$
\end{lemma}

\begin{proof}
Note that $B$ has $\Bbbk$-basis $\{1, x, y, yx\}$, and
$$\rad(B)=\Bbbk x\oplus \Bbbk y\oplus \Bbbk yx,\qquad \rad^2(B)=\Bbbk yx.$$ Suppose first that $\lambda_1\lambda_2=1$. Then $z:=x-\lambda_1y$ satisfies
$$zx=x^2-\lambda_1yx=0,\qquad zy=xy-\lambda_1y^2=(1-\lambda_1\lambda_2)yx=0.$$
Thus $z,yx\in\operatorname{soc}(B)$. Since $z\notin \rad^2(B)$ whereas $yx\in \rad^2(B)$, they are linearly independent. Consequently, $\mathrm{dim}_{\Bbbk}\operatorname{soc}(B)\ge2$. Therefore, if $B$ is symmetric, then $\lambda_1\lambda_2\neq1.$ Conversely, suppose that $\lambda_1\lambda_2\neq1$. Define a linear form $\varphi:B\to \Bbbk$ by
$$\varphi(yx)=1,\qquad\varphi(1)=\varphi(x)=\varphi(y)=0.$$
Then the associative symmetric bilinear form $\langle u,v\rangle=\varphi(uv)$ is nondegenerate, so $B$ is symmetric.

Finally, assume that $B$ is symmetric. Then $\lambda_1\lambda_2\neq1$. Since $\Bbbk$ is algebraically closed, there exists $0\neq c\in \Bbbk$ satisfying
$c^2(\lambda_1\lambda_2-1)=1.$ Let $X=x$ and $Y=c(x-\lambda_1y)$ be two elements in $B$. One easily verifies that $$X^2=Y^2=\lambda_1yx,\qquad XY=YX=0.$$ Hence the assignment $x\mapsto X$, and $y\mapsto Y$
defines a homomorphism $A\longrightarrow B$. Moreover, the images of $X$ and $Y$ are linearly independent in $\rad(B)/\rad^2(B)$, so they generate $B$. Therefore the homomorphism is surjective. Since both $A$ and $B$ are $4$-dimensional, it is an isomorphism. Thus $A\cong B.$
\end{proof}

By comparing the tubes in the stable Auslander-Reiten quiver of a symmetric stably biserial algebra with the faces of its corresponding Brauer graph, Antipov and Zvonareva established in \cite[Section 4.2]{AZ2} the following result.

\begin{Prop}\label{prop:stbA-tube-and-face}
Let $A$ be a representation-infinite symmetric stably biserial algebra with Brauer graph $\Gamma$. Then, for any face of $\Gamma$ with perimeter $p > 2$, the structure of the stable Auslander-Reiten quiver of $A$ reflects this combinatorial data. If $p$ is odd, the quiver contains a tube of rank $p$ that is invariant under the syzygy functor $\Omega_A$. If $p$ is even, it contains a pair of tubes, each of rank $p/2$, which are exchanged by $\Omega_A$.
\end{Prop}

Specifically, for a face $F=(h_1,h_2,\cdots,h_{p})$, let us denote by $\mathcal{M}_F$ the set of modules in the mouth of the $p$-tube or two $p/2$-tubes, corresponding to $F$. They can be constructed as follows: take two consecutive half-edges $h_i$ and $h_{i+1}=\iota\rho(h_{i})$ from the face $F$, they correspond to two vertices $\bar{h_i}$ and $\bar{h_{i+1}}$ in the quiver of $A$ and an arrow $\alpha_{h_i}: \bar{h_i}\rightarrow \bar{h_{i+1}}$ (resp. a path $C_{\alpha_{\rho^{-1}(h_{i+1})}}: \bar{h_{i+1}}\rightarrow \bar{h_{i+1}}$) if $s(h_i)$ is not truncated (resp. if $s(h_i)$ is truncated); the modules in $\mathcal{M}_F$ are of the form $M_i:=Ae_{\bar{h_i}}/A\alpha_{h_i}$ (resp. $M_i:=Ae_{\bar{h_i}}/AC_{\alpha_{\rho^{-1}(h_{i+1})}}$). Thus, $\mathcal{M}_F=\{M_1,\cdots,M_p\}$. Note that the modules in the union of all $\mathcal{M}_F$ are exactly the simple modules whose projective covers are uniserial and the maximal uniserial quotients of indecomposable non-uniserial projective modules.

It should be noted that if one considers only Brauer graph algebras (that is, symmetric special biserial algebras), then for a face in $\Gamma$ with perimeter $1$ or $2$, Proposition \ref{prop:stbA-tube-and-face} also holds. More precisely, if $p=1$, then such a face corresponds to a tube of rank $1$ in the stable Auslander-Reiten quiver, and if $p=2$, then such a face corresponds to a pair of two tubes of rank $1$ in the stable Auslander-Reiten quiver which are exchanged by the syzygy functor. For instance, the $A_{\Gamma_1}$ in Example \ref{ex:two-exceptional-band} illustrates this correspondence for a face of perimeter $2$, which yields a pair of tubes that are not stable under the syzygy functor.

However, this correspondence fails for the broader class of symmetric stably biserial algebras. As a counterexample, consider a deformed loop in a symmetric stably biserial algebra $A$ (which gives rise to a face of perimeter $1$), as shown below:
$$
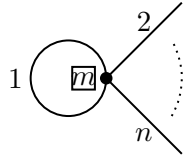
\begin{tikzpicture}
\draw (-0.5,0) circle (0.5);
\fill (0,0) circle (0.5ex);
\node at(-1.2,0) {$1$};
\node at(-0.3,0) {$m$};
\draw[-] (0,0) -- (1,1);
\draw[-] (0,0) -- (1,-1);
\draw (-0.45,-0.15) rectangle (-0.15,0.15);
\node at(0.5,0.75) {$2$};
\node at(0.5,-0.75) {$n$};
\draw[dotted] (0.866,-0.5) arc (-30:30:1);
\end{tikzpicture}$$
Then there exist two tubes of rank 1 which are not stable under the syzygy functor $\Omega_A$; indeed, $\Omega_A$ sends one to the other. The indecomposable modules at the mouths of these homogeneous tubes are as follows:

$$\begin{tikzpicture}
	\begin{scope}[xshift=-60pt]
		\node at(0,0) {$1$};
\node at(0,-0.3) {$2$};
\draw[dotted] (0,-0.5) -- (0,-0.9);
\node at(0,-1.1) {$n$};
\node at(0,-1.4) {$1$};
\node at(0,-1.7) {$1$};
\draw[dotted] (0,-1.9) -- (0,-2.3);
\node at(0,-2.5) {$1$};
\node at(0,-2.8) {$2$};
\draw[dotted] (0,-3) -- (0,-3.4);
\node at(0,-3.6) {$n$};
\node at(0,-3.9) {$1$};
	\end{scope}
		\begin{scope}[xshift=60pt]
		\node at(0,0) {$1$};
\node at(0,-0.3) {$2$};
\draw[dotted] (0,-0.5) -- (0,-0.9);
\node at(0,-1.1) {$n$};
\node at(0,-1.4) {$1$};
\node at(0,-1.7) {$1$};
\draw[dotted] (0,-1.9) -- (0,-2.3);
\node at(0,-2.5) {$1$};
\node at(0,-2.8) {$2$};
\draw[dotted] (0,-3) -- (0,-3.4);
\node at(0,-3.6) {$n$};
\node at(0,-3.9) {$1$};
\node at(0.7,-2) {$t$\;.};
\draw[dashed] (0.15,-3.9) arc (-15:15:7.5);
	\end{scope}
	\node at(0,-2) {\text{and}};
\end{tikzpicture}$$
Note that the former one is a string module and the latter one is an exceptional band module. We will study the stable Hom-spaces from these modules in Section \ref{app:lem}.

\subsection{Stable Grothendieck groups}\label{subsec:stable-K0}
\

We recall the \textit{stable Grothendieck group} in the form used in \cite{MV90}. For a finite-dimensional algebra $\varLambda$, let $K_0(\varLambda\text{-}\mathrm{mod})$ be the abelian group generated by isomorphism classes $[M]$ of finite-dimensional modules, with relations $[M]=[L]+[N]$ for short exact sequences $0\to L\to M\to N\to0$. Then
\[
K_0^{\mathrm{st}}(\varLambda)
=K_0(\varLambda\text{-}\mathrm{mod})/
\langle[P]\mid P\text{ is projective}\rangle.
\]
Equivalently, this is the cokernel of the \textit{Cartan map} from the split Grothendieck group of finitely generated projective modules to $K_0(\varLambda\text{-}\mathrm{mod})$. If $S_1,\ldots,S_n$ are the simple modules and $P_j$ is the projective cover of $S_j$, the \textit{Cartan matrix} $C_\varLambda=(c_{ij})$ has entries $c_{ij}=[P_j:S_i]$, and
\[
K_0^{\mathrm{st}}(\varLambda)\cong
\mathbb Z^n/C_\varLambda\mathbb Z^n.
\]
In particular, for a basic local algebra over our algebraically closed field $\Bbbk$, this specializes to
\[
K_0^{\mathrm{st}}(\varLambda)\cong
\mathbb Z/(\mathrm{dim}_{\Bbbk}\varLambda)\mathbb Z.
\]

Recall that a \textit{node} (see \cite{MV90}) is a simple module $S$ which is neither projective nor injective and for which the middle term of the almost split sequence starting at $S$ is projective. We use the following standard results.

\begin{Prop}\textnormal{(see \cite{MV90})}
\label{prop:stable-K0}
Let $A$ and $B$ be finite-dimensional algebras without nodes and without semisimple summands. A stable equivalence between $A$ and $B$ induces an isomorphism $K_0^{\mathrm{st}}(A)\cong K_0^{\mathrm{st}}(B)$.
\end{Prop}

\begin{Lem}\textnormal{(cf. \cite[Chapter X, Proposition 1.8]{ARS})}
\label{lem:selfinjective-no-nodes}
An indecomposable self-injective algebra which has a node is a Nakayama algebra with radical square zero. In particular, an indecomposable self-injective algebra $\varLambda$ with $(\rad\varLambda)^2\ne0$ has no nodes.
\end{Lem}

\subsection{The Ziegler spectrum and stable equivalences}\label{subsec:ziegler}
\

Here and in Section~\ref{sec:local-case}, modules may be infinite-dimensional. We write $\varLambda\text{-}\mathrm{Mod}$ for all left $\varLambda$-modules and retain $\varLambda\text{-}\mathrm{mod}$ for finite-dimensional modules. Direct limits are taken over directed index sets. Our purpose is to explain how to determine whether a generic module belongs to the closure of a homogeneous tube.

\subsubsection{Purity and closure under module constructions}
\

For left $\varLambda$-modules $L,M,N$, a short exact sequence
\[
0\longrightarrow L\xrightarrow{i}M\xrightarrow{q}N\longrightarrow0
\]
is \emph{pure exact} if, for every finite-dimensional left $\varLambda$-module $C$, the sequence
\[
0\longrightarrow\mathrm{Hom}_\varLambda(C,L)
\xrightarrow{i_*}\mathrm{Hom}_\varLambda(C,M)
\xrightarrow{q_*}\mathrm{Hom}_\varLambda(C,N)\longrightarrow0
\]
is exact; see \cite[Section 2.1.1]{Prest} and \cite[Section 1]{Krause97}. In this situation $i$ is called a \emph{pure embedding}, and $i(L)$ a \emph{pure submodule} of $M$.

A module $L$ is \emph{pure-injective} if every pure exact sequence $0\to L\xrightarrow{i}M\xrightarrow{q}N\to0$, with arbitrary modules $M,N$, splits; see \cite[Section 4.3.1]{Prest}.

A full subcategory of $\varLambda\text{-}\mathrm{Mod}$, closed under isomorphisms, is \emph{definable} if it is closed under products, direct limits and pure submodules \cite[Section 3.4.1]{Prest}. For a class of modules $\mathcal C$, write $\operatorname{Def}_\varLambda(\mathcal C)$ for the smallest such subcategory containing $\mathcal C$, namely the intersection of all definable subcategories containing $\mathcal C$. Split exact sequences are pure exact, so direct summands are pure submodules. In particular, definable subcategories are closed under direct summands.

\subsubsection{The closure criterion}
\

The points of the \emph{Ziegler spectrum} $\operatorname{Zg}\varLambda$ are the isomorphism classes of nonzero indecomposable pure-injective left $\varLambda$-modules. Its topology is characterized by the fact that its closed subsets are exactly
\[
\mathcal D\cap\operatorname{Zg}\varLambda,
\qquad\mathcal D\text{ a definable subcategory of }
\varLambda\text{-}\mathrm{Mod};
\]
see \cite[Section 5.1]{Prest} and \cite[Section 1]{Krause97}. We identify points with representative modules. Thus, for every $\mathcal C\subseteq\operatorname{Zg}\varLambda$,
\[
\overline{\mathcal C}^{\,\operatorname{Zg}\varLambda}
=\operatorname{Def}_\varLambda(\mathcal C)\cap\operatorname{Zg}\varLambda.
\]

A module $G$ is \emph{generic} if it is indecomposable, infinite-dimensional over $\Bbbk$, and has finite length under the evaluation action of $\operatorname{End}_\varLambda(G)$ on $G$. The finite-length condition is called \emph{endofiniteness}. Endofinite modules are pure-injective \cite[Section 3]{Krause97}; hence generic modules are points of $\operatorname{Zg}\varLambda$. Finite-dimensional indecomposable modules are also points \cite[Section 8.1.1]{Prest}, so we may regard a homogeneous tube $\mathcal T$ as the set of its finite-dimensional indecomposable modules in this space. For a generic module $G$, the preceding closure description gives the criterion
\[
G\in\overline{\mathcal T}^{\,\operatorname{Zg}\varLambda}
\quad\Longleftrightarrow\quad
G\in\operatorname{Def}_\varLambda(\mathcal T).
\]

\begin{Ex1}[A generic module in the closure of a Kronecker tube]\normalfont
\label{ex:Zg-tube}
Let $A=\Bbbk(1\rightrightarrows2)$, with arrows $\alpha,\beta:1\to2$. For any $\Bbbk[t]$-module $V$, define the $A$-module $E(V)$ by placing $V$ at both vertices and taking
\[
\alpha_V:V\to V,\quad v\mapsto v,
\qquad
\beta_V:V\to V,\quad v\mapsto tv.
\]
A $\Bbbk[t]$-linear map $f:V\to W$ gives the morphism $E(f):E(V)\to E(W)$ whose component at each vertex is $f$. Conversely, commutation with $\alpha$ forces the two components of an $A$-homomorphism to coincide, and commutation with $\beta$ makes that common map $\Bbbk[t]$-linear. Thus
\[
\mathrm{Hom}_A(E(V),E(W))\cong\mathrm{Hom}_{\Bbbk[t]}(V,W).
\]
Products and direct limits of representations are computed at each vertex, so $E$ preserves both operations. It also preserves direct-sum decompositions.

For $c\in \Bbbk$ and $r\geq1$, put
\[
U_r(c)=E(\Bbbk[t]/((t-c)^r)),\qquad
\mathcal T_c=\{U_r(c):r\geq1\}.
\]
The map $\beta$ on $U_r(c)$ is a single Jordan block with eigenvalue $c$. These modules form a homogeneous tube, with $U_r(c)$ of quasi-length $r$; see \cite[Section 8.1.2]{Prest}. We show that
\[
G=E(\Bbbk(t))\in\overline{\mathcal T_c}^{\,\operatorname{Zg}A}
\qquad\text{for every }c\in \Bbbk.
\]

First, $\operatorname{End}_A(G)\cong\operatorname{End}_{\Bbbk[t]}(\Bbbk(t))=\Bbbk(t)$. Indeed, a $\Bbbk[t]$-linear endomorphism of $\Bbbk(t)$ commutes with multiplication by every nonzero polynomial and its inverse, so it is $\Bbbk(t)$-linear. The endomorphism ring is a field, hence $G$ is indecomposable. Under this ring action $G=\Bbbk(t)\oplus \Bbbk(t)$ has length two, whereas $\mathrm{dim}_{\Bbbk}G=\infty$. Thus $G$ is generic and is a point of $\operatorname{Zg}A$.

Fix $c\in \Bbbk$ and form the $\Bbbk[t]$-module
\[
W_c=\prod_{r\geq1}\Bbbk[t]/((t-c)^r),
\qquad
\eta=(\overline{1})_{r\geq1}\in W_c.
\]
Here the $r$th coordinate $\overline{1}$ is the residue class of $1$ in $\Bbbk[t]/((t-c)^r)$.
If $0\ne f\in \Bbbk[t]$, choose $r$ greater than the multiplicity of $t-c$ in $f$. Then the $r$th coordinate of $f\eta$ is nonzero. Thus $\eta$ is not annihilated by any nonzero polynomial, and its image $\eta/1$ in the localization $V_c=\Bbbk(t)\otimes_{\Bbbk[t]} W_c$ is nonzero.

For completeness, $V_c$ is a direct limit of copies of $W_c$: denote by $I$ the directed set of monic polynomials in $\Bbbk[t]$ ordered by divisibility, and let $\{B_{f},\phi^{f}_{g}\}$ be a direct system of $\Bbbk[t]$-modules over $I$, where $B_{f}=\Bbbk[t]$ for each $f\in I$ and $\phi^{f}_{g}:B_{f}\rightarrow B_{g}$ is the $\Bbbk[t]$-linear homomorphism mapping $1$ to $\frac{g}{f}$ for each $f\mid g$. Then $\Bbbk(t)$ is the direct limit of the direct system $\{B_{f},\phi^{f}_{g}\}$, where the map $\psi_{f}$ from $B_{f}$ to $\Bbbk(t)$ is given by $\psi_{f}(1)=\frac{1}{f}$. Since the functor $-\otimes_{\Bbbk[t]}W_c$ preserves direct limits, $V_c$ is a direct limit of copies of $\Bbbk[t]\otimes_{\Bbbk[t]}W_c\cong W_c$. As $E$ preserves products and these limits, both
\[
E(W_c)=\prod_{r\geq1}U_r(c)
\quad\text{and}\quad E(V_c)
\]
belong to $\operatorname{Def}_A(\mathcal T_c)$. Choose a $\Bbbk(t)$-basis of $V_c$ containing $\eta/1$. The inclusion and coordinate projection associated with this basis give
\[
j:\Bbbk(t)\longrightarrow V_c,\quad a\longmapsto a\eta/1,
\qquad
p:V_c\longrightarrow \Bbbk(t),\quad p\circ j=\mathrm{id}_{\Bbbk(t)}.
\]
Applying $E$ shows that $G=E(\Bbbk(t))$ is a direct summand of $E(V_c)$. It follows that $G\in\operatorname{Def}_A(\mathcal T_c)$ and hence $G\in\overline{\mathcal T_c}^{\,\operatorname{Zg}A}$. Equivalently, every open neighbourhood of $G$ contains some $U_r(c)$. The module $G$ is independent of $c$, so distinct tubes $\mathcal T_c$ have a common generic module in their closures. The same construction will be used for the local algebras in Lemma~\ref{lem:loc-generic}.
\end{Ex1}

\subsubsection{Krause's homeomorphism}
\

Let $\mathscr P_\varLambda$ be the finite set of isomorphism classes of indecomposable projective $\varLambda$-modules. We give $\operatorname{Zg}\varLambda\setminus\mathscr P_\varLambda$ the subspace topology.

\begin{Thm}\textnormal{(see \cite{Krause97})}\label{thm:Krause-homeomorphism}
Let $A,B$ be finite-dimensional algebras and let $F:B\text{-}\underline{\mathrm{mod}} \to A\text{-}\underline{\mathrm{mod}}$ be a stable equivalence. Then the following statements hold.
\begin{enumerate}
\item $F$ induces a homeomorphism
\[
\widehat F:\operatorname{Zg}B\setminus\mathscr P_B
\xrightarrow{\sim}\operatorname{Zg}A\setminus\mathscr P_A
\]
extending its action on finite-dimensional nonprojective indecomposable modules.

\item The homeomorphism $\widehat F$ restricts to a bijection between the isomorphism classes of generic $B$-modules and generic $A$-modules.

\item If two sets of finite-dimensional nonprojective indecomposable $B$-modules have a common generic module in their Ziegler closures, then their images under $F$ have a common generic module in their Ziegler closures over $A$.
\end{enumerate}
\end{Thm}

\begin{proof}
\noindent(1)
For $\varLambda=A$ (or $B$), the subcategory $\operatorname{add}\varLambda$ of finitely generated projective modules is covariantly finite, and its direct-limit closure consists of direct sums of indecomposable projective modules; see
\cite[Proposition 3.2]{Krause97}. Thus the points of $\operatorname{Zg}\varLambda$ belonging to this closure are precisely $\mathscr P_\varLambda$.
By \cite[Proposition 2.2(4)]{Krause97}, $\operatorname{Zg}\varLambda\setminus\mathscr P_\varLambda$ is homeomorphic to the Ziegler spectrum of the locally finitely presented category obtained by adjoining direct limits to $\varLambda\text{-}\underline{\mathrm{mod}}$. Extending $F$ to direct limits by \cite[Lemma 1.1]{Krause97} and using these identifications yields the required homeomorphism $\widehat F$.

\smallskip
\noindent(2)
The restriction of $\widehat F$ to generic modules is the bijection in \cite[Corollary 3.5]{Krause97}.

\smallskip
\noindent(3)
Let $\mathcal C_1,\mathcal C_2$ be the given sets, and let $G$ be a generic $B$-module belonging to both Ziegler closures. Since $G$ is infinite-dimensional, $G\notin\mathscr P_B$. As each $\mathcal C_i$ consists of nonprojective points, the definition of the subspace topology gives
\[
G\in
\overline{\mathcal C_i}^{\,\operatorname{Zg}B}
\cap(\operatorname{Zg}B\setminus\mathscr P_B)
=
\overline{\mathcal C_i}^{\,\operatorname{Zg}B\setminus\mathscr P_B}
\qquad(i=1,2).
\]
The homeomorphism $\widehat F$ preserves closures in these subspaces and agrees with $F$ on $\mathcal C_i$. Hence
\[
\widehat F(G)\in
\overline{F(\mathcal C_i)}^{\,\operatorname{Zg}A\setminus\mathscr P_A}
\subseteq
\overline{F(\mathcal C_i)}^{\,\operatorname{Zg}A}
\qquad(i=1,2).
\]
By part~(2), $\widehat F(G)$ is generic, proving the assertion.
\end{proof}

\subsection{Stable equivalences and syzygy functors}
\label{subsec:stable-equivalences-symmetric}
\

The following two propositions recall results of \cite[Chapter~X]{ARS} on the compatibility of stable equivalences with syzygies and Auslander--Reiten translations. In the symmetric case, we include a proof making the naturality of the compatibility with syzygies explicit.

\begin{Prop}\textnormal{(see \cite[Chapter~X]{ARS})}\label{prop:stable-equivalence-self-injective}
Let $A$ and $B$ be self-injective algebras with no blocks of
Loewy length $2$, and let $F:B\text{-}\underline{\mathrm{mod}}
\longrightarrow A\text{-}\underline{\mathrm{mod}}$ be a $\Bbbk$-linear equivalence. For every $X\in B\text{-}\underline{\mathrm{mod}}$, there are
isomorphisms
\[
F(\Omega_BX)\cong\Omega_A(FX),
\qquad
F(\tau_BX)\cong\tau_A(FX).
\]
Moreover, $F$ induces an isomorphism of stable Auslander--Reiten quivers as translation quivers. In particular, it sends stable tubes to stable tubes of the same rank and preserves their mouths and the quasi-lengths of their vertices.
\end{Prop}

\begin{proof}
The isomorphisms involving syzygies follow from
\cite[Chapter~X, Proposition~1.12(a)]{ARS}.
The assertions concerning Auslander--Reiten translations
and stable Auslander--Reiten quivers follow from
\cite[Chapter~X, Corollaries~1.7(a) and~1.9(b)]{ARS}.
The rank, mouth, and quasi-lengths of a stable tube are
determined by its translation-quiver structure.
\end{proof}

We note that the assumption excluding blocks of Loewy length $2$ cannot be omitted in general; see \cite[Remark~1.3(1)]{LL2}.

For symmetric algebras, the compatibility with syzygies can be expressed as a natural isomorphism of functors, as we now explain.

\begin{Prop}\textnormal{(see \cite[Lemma~I.1.3(1)]{RVdB})}\label{prop:stable-equivalence-syzygy}
Let $A$ and $B$ be symmetric algebras, and let $F:B\text{-}\underline{\mathrm{mod}} \longrightarrow A\text{-}\underline{\mathrm{mod}}$
be a $\Bbbk$-linear equivalence. Then there is a natural isomorphism $F\Omega_B\cong\Omega_AF$. Consequently, there is a natural isomorphism $F\tau_B\cong\tau_AF$.
\end{Prop}

\begin{proof}
Using the isomorphism in the proof of \cite[Chapter~X, Lemma~1.11]{ARS},
Auslander--Reiten duality \cite[pp.~392--393]{ARS}, and the identity $\tau_{\varLambda}\cong\Omega_{\varLambda}^{2}$ for symmetric algebras \cite[Chapter~IV, Proposition~3.8(c)]{ARS}, full faithfulness of $F$ gives
\[
\begin{aligned}
\sHom_A(FY,F\Omega_BX)
&\cong \sHom_B(Y,\Omega_BX)\\
&\cong \operatorname{Ext}^1_B(Y,\Omega_B^2X)\\
&\cong D\sHom_B(X,Y)\\
&\cong D\sHom_A(FX,FY)\\
&\cong \operatorname{Ext}^1_A(FY,\Omega_A^2FX)\\
&\cong \sHom_A(FY,\Omega_AFX),
\end{aligned}
\]
naturally in $X,Y\in B\text{-}\underline{\mathrm{mod}}$. Since $F$ is essentially surjective, the Yoneda lemma yields isomorphisms $F\Omega_BX\cong\Omega_AFX$. Naturality in $X$ ensures that these assemble into a natural isomorphism $F\Omega_B\cong\Omega_AF$. The natural isomorphism $F\tau_B\cong\tau_AF$ follows from $\tau_B\cong\Omega_B^2$ and $\tau_A\cong\Omega_A^2$.
\end{proof}

Note that, for a symmetric algebra $\varLambda$, the syzygy functor $\Omega_{\varLambda}$ is a Serre functor on $\varLambda\text{-}\underline{\mathrm{mod}}$. The preceding proof is the usual Yoneda argument that equivalences commute with Serre functors; compare \cite[Lemma~I.1.3(1)]{RVdB}.

\begin{remark}
Ohnuki states in the conference report \cite[Theorem~2]{Ohnuki05} that a stable equivalence between finite-dimensional self-injective algebras admits the structure of a triangle equivalence if and only if it commutes with the syzygy functors up to natural isomorphism.
As a corollary, he states that every stable equivalence between finite-dimensional symmetric algebras admits such a structure \cite[Corollary~3]{Ohnuki05}. 
\end{remark}

\subsection{Stable equivalences of Morita type}\label{subsec:stb-and-center}
\

According to \cite[Proposition~5.2]{BHS}, an algebra stably equivalent to an indecomposable symmetric algebra with nonzero radical square must be weakly symmetric. However, we even do not know a concrete example where a symmetric algebra is stably equivalent to some non-symmetric algebra over an algebraically closed field, although there do exist such examples over a non-algebraically closed field (see \cite{OTY}).

Therefore, in order to simplify our discussion, we consider those stable equivalences which preserve the property of being symmetric. This leads us to consider stable equivalences of Morita type. We recall their basic definition below and refer the reader to~\cite[Chapter 5]{Z} for further details.

\begin{Def}
Let $A$ and $B$ be two $\Bbbk$-algebras. A stable equivalence of Morita type between $A$ and $B$ consists of bimodules
\[
{}_A M_B \quad \text{and} \quad {}_B N_A
\]
such that:
\begin{enumerate}
\item $M$ is projective as a left $A$-module and as a right $B$-module;
\item $N$ is projective as a left $B$-module and as a right $A$-module;
\item There exist projective bimodules $P$ (an $A$-$A$-bimodule) and $Q$ (a $B$-$B$-bimodule) such that
  \[
  M \otimes_B N \cong A \oplus P \quad \text{as } A\text{-}A\text{-bimodules,}
  \]
and
  \[
  N \otimes_A M \cong B \oplus Q \quad \text{as } B\text{-}B\text{-bimodules.}
  \]
\end{enumerate}

\noindent In this situation, the tensor functor{
\[
M \otimes_B - \; : \; B\text{-}\mathrm{mod}\longrightarrow A\text{-}\mathrm{mod}
\]
induces a stable equivalence $B\text{-}\underline{\mathrm{mod}}\longrightarrow A\text{-}\underline{\mathrm{mod}}$, with quasi-inverse induced by
\[
N \otimes_A - \; : \; A\text{-}\mathrm{mod} \longrightarrow B\text{-}\mathrm{mod}.
\]
}\end{Def}

Clearly, the above functor $M \otimes_B - \; : \; B\text{-}\mathrm{mod}\longrightarrow A\text{-}\mathrm{mod}$ is exactand preserves projective modules. Its induced functor on the stable categories therefore commutes with the syzygy functors $\Omega_A$ and $\Omega_B$ up to natural isomorphism. In particular, if both $A$ and $B$ are self-injective algebras, then exactness makes the induced stable equivalence $B\text{-}\underline{\mathrm{mod}}\longrightarrow A\text{-}\underline{\mathrm{mod}}$ a triangle equivalence. Moreover, a stable equivalence of Morita type preserves the property of being symmetric and indecomposability of algebras (see \cite[Proposition 2.1 and Corollary 2.4]{L}).

\section{Stable Hom-spaces and deformed loops}\label{app:lem}
We first prove a stable Hom-dimension formula for $2$-periodic modules over self-injective algebras. We then show that, over a symmetric stably biserial algebra, the stable Hom-space from a string module associated with a deformed loop to any string module has dimension at most $2$. Combining this bound with the formula yields Corollary~\ref{lem:equal-dimension}.

\subsection{Stable Hom-dimensions for $2$-periodic modules}
\label{subsec:selfinjective-Hom}
\

Let $\varLambda$ be a self-injective algebra. We take syzygies using projective covers. For finite-dimensional $\varLambda$-modules $U,V$, recall the standard identification (see for example in \cite[Proposition~4.13.7]{Linckelmann18})
\[
\operatorname{Ext}^{n}_{\varLambda}(U,V)
\cong\sHom_{\varLambda}(\Omega_{\varLambda}^{n}U,V)
\qquad(n\geq1).
\]

\begin{Prop}\label{prop:periodic-stable-Hom}
Let $\varLambda$ be a self-injective algebra and let $X$ be a finite-dimensional $\varLambda$-module satisfying
$\Omega_{\varLambda}^{2}X\cong X$ and $\top(\Omega_{\varLambda}X)\cong\top(X)$.  For every finite-dimensional $\varLambda$-module $Y$ and every integer $n\geq1$, we have
\[
\mathrm{dim}_{\Bbbk}\sHom_{\varLambda}(X,Y)
=\mathrm{dim}_{\Bbbk}\sHom_{\varLambda}(\Omega_{\varLambda}X,Y)=\mathrm{dim}_{\Bbbk}\operatorname{Ext}^{n}_{\varLambda}(X,Y)
\]
\end{Prop}

\begin{proof}
Let $P$ be the projective cover of $X$. Since $\top(\Omega_{\varLambda} X)\cong\top(X)$ and $\Omega_{\varLambda}^2X\cong X$, the projective covers give exact sequences
\[
0\longrightarrow\Omega_{\varLambda} X\longrightarrow P\longrightarrow X\longrightarrow0,
\qquad
0\longrightarrow X\longrightarrow P\longrightarrow\Omega_{\varLambda} X\longrightarrow0.
\]
Applying $\mathrm{Hom}_{\varLambda}(-,Y)$ and using the recalled identification of Ext with stable Hom, we obtain
\[
\begin{aligned}
\mathrm{dim}_{\Bbbk}\mathrm{Hom}_{\varLambda}(X,Y)
+\mathrm{dim}_{\Bbbk}\mathrm{Hom}_{\varLambda}(\Omega_{\varLambda} X,Y)
-\mathrm{dim}_{\Bbbk}\mathrm{Hom}_{\varLambda}(P,Y)
&=\mathrm{dim}_{\Bbbk}\sHom_{\varLambda}(\Omega_{\varLambda} X,Y),\\
\mathrm{dim}_{\Bbbk}\mathrm{Hom}_{\varLambda}(\Omega_{\varLambda} X,Y)
+\mathrm{dim}_{\Bbbk}\mathrm{Hom}_{\varLambda}(X,Y)
-\mathrm{dim}_{\Bbbk}\mathrm{Hom}_{\varLambda}(P,Y)
&=\mathrm{dim}_{\Bbbk}\sHom_{\varLambda}(X,Y).
\end{aligned}
\]
Subtracting gives the equality of the two stable Hom dimensions. The remaining equalities follow from $\Omega_{\varLambda}^2X\cong X$ and the same Ext--stable Hom identification for every $n\geq1$.
\end{proof}

\subsection{String modules associated with deformed loops}
\label{subsec:deformed-loop-Hom}
\

\begingroup
\emergencystretch=3em

Let $\varLambda = \Bbbk Q/I$ be a symmetric stably biserial algebra and let $X$ be a string $\varLambda$-module.
Recall that $X$ is indecomposable and can be viewed as a module over the associated string algebra.
The basis of $X$ given by the construction in Definition~\ref{def:string-module} is called a \textit{diagrammatic basis} of $X$, following \cite{A,AZ1}.
Each arrow of $Q$ sends a diagrammatic basis element to zero or to another basis element, and the nonzero images of distinct basis elements are distinct.

Using this basis, we depict $X$ by a finite quiver $D$ of type $A_n$.
The vertices of $D$ are in bijection with the diagrammatic basis elements and are labelled by the corresponding vertices of $Q$.
We call the labelled quiver $D$ the \textit{diagram} of $X$.
Note that distinct vertices of $D$ may be labelled by the same vertex of $Q$.
At each vertex of $D$, two incoming arrows or two outgoing arrows (if they exist) have different labels.
Moreover, every directed path in $D$ is labelled by a nonzero path in the associated string algebra.

\endgroup
 
We now bound the dimension of the stable Hom-space from a string module associated with a deformed loop to an arbitrary string module. The terminology and the description of the two mouth modules are those of Subsection~\ref{subsec:symmetric-stably-biserial}.

\begin{Lem}\label{lem:deformed-loop-bound}
Let $\varLambda=\Bbbk Q/I$ be a symmetric stably biserial algebra with the presentation in Theorem~\ref{thm:sym-StBA}. Let $\alpha:v\to v$ be a deformed loop and let $X=\varLambda e_v/\varLambda\alpha$
be the associated string module at the mouth of a homogeneous tube. For every string $\varLambda$-module $Y$, we have
\[
\mathrm{dim}_{\Bbbk}\sHom_{\varLambda}(X,Y)\leq2.
\]
\end{Lem}

\begin{proof}[Proof]
    Without loss of generality, we assume that $Y$ is nonprojective. Then $\soc(\varLambda)Y=0$, since $\varLambda$ is self-injective and $Y$ is indecomposable.
    For the projective $\varLambda$-module $P=\varLambda e_v$, set $x=e_v+\varLambda\alpha$ and $t=t_\alpha$. Write
    \[
    C_\alpha^{m(C_\alpha)}=\alpha p.
    \]
    By Theorem~\ref{thm:sym-StBA}, we have $\alpha(\alpha-tp)=0$ and hence there is a homomorphism
    \[
    \varphi:X\longrightarrow P,\qquad x\mapsto\alpha-tp.
    \]
    
    For every diagrammatic basis element $z$ of $Y$ labelled by $v$ and satisfying $\alpha z=0$, consider the homomorphism $f_z:X\to Y,\ x\mapsto z$. These maps form a basis of $\Hom_{\varLambda}(X,Y)$. Indeed, a homomorphism $f:X\to Y$ is determined by $f(x)$, which can be any element of $e_vY$ annihilated by $\alpha$. The arrow $\alpha$ sends diagrammatic basis elements to zero or to distinct diagrammatic basis elements, so the kernel of its action on $e_vY$ is spanned by the basis elements it annihilates.

    For a diagrammatic basis element $w\in e_vY$, consider the homomorphism 
    \[q_w:P\to Y,\ e_v\mapsto w.\]
    The elements $\alpha w$ and $pw$ of $e_vY$ are each either zero or a diagrammatic basis element of $Y$. Since $\alpha^2$ and $\alpha p$ belong to $\soc(\varLambda)$, we have $\alpha(\alpha w)=\alpha(pw)=0$. With $f_0=0$, we obtain
    \[
    f_{\alpha w}-t f_{pw}=q_w\varphi,\qquad
    \underline{f_{\alpha w}}=t\,\underline{f_{pw}}.
    \]

    We next consider those $f_z$ for which $z$ is an internal vertex (that is, a vertex other than an endpoint) in the diagram of $Y$. We now show that each such $z$, labelled by $v$ and satisfying $\alpha z=0$, is the target of an arrow labelled by $\alpha$.
    
    Consider the local structure of the quiver of $\varLambda$ at the vertex $v$:
    % https://q.uiver.app/#q=WzAsMyxbMSwwLCJ2Il0sWzAsMCwiXFxidWxsZXQiXSxbMiwwLCJcXGJ1bGxldCJdLFswLDAsIlxcYWxwaGEiXSxbMSwwLCJcXHNpZ21hKFxcYWxwaGEpIiwyXSxbMCwyLCJcXHNpZ21hXnstMX0oXFxhbHBoYSkiLDJdXQ==
    \[\begin{tikzcd}
      \bullet & v & \bullet
      \arrow["{\sigma(\alpha)}"', from=1-1, to=1-2]
      \arrow["\alpha", from=1-2, to=1-2, loop, in=55, out=125, distance=10mm]
      \arrow["{\sigma^{-1}(\alpha)}"', from=1-2, to=1-3]
    \end{tikzcd}.\]
    Note that the unlabelled vertices might be $v$ and the arrows $\sigma(\alpha)$ and $\sigma^{-1}(\alpha)$ may coincide.

    In the diagram of $Y$, if $z$ is an internal source, it has an outgoing arrow labelled by $\alpha$, contradicting $\alpha z=0$, whereas if it is an internal sink, it has an incoming arrow labelled by $\alpha$. If $z$ has one incoming and one outgoing arrow and neither is labelled by $\alpha$, they would form a directed path labelled by $\sigma^{-1}(\alpha)\sigma(\alpha)$. Since $\sigma(\sigma^{-1}(\alpha))=\alpha\ne\sigma(\alpha)$, the relations in Theorem~\ref{thm:sym-StBA} imply that this path is zero in $\varLambda$ unless $\sigma^{-1}(\alpha)$ is itself a deformed loop. In that exceptional case, $\sigma^{-1}(\alpha)=\sigma(\alpha)$, and $(\sigma(\alpha))^2$ belongs to $\soc(\varLambda)$, so the path still acts as zero on $Y$. Thus this configuration cannot occur. The only remaining possibility is that $z$ is the target of an arrow labelled by $\alpha$.

    For such an internal vertex, write $z=\alpha w$. If $pw=0$, then $\underline{f_z}=\underline{f_{\alpha w}}=t\,\underline{f_{pw}}=0$. Otherwise, put $z'=pw$; then $\alpha z'=0$ and $\underline{f_z}=t\,\underline{f_{z'}}$. The corresponding part of the diagram has the form
    \[
    z\xleftarrow{\alpha}w
    \underset{p}{\xrightarrow{\sigma^{-1}(\alpha)}\cdots
    \xrightarrow{\sigma(\alpha)}}
    z'.
    \]
    If $z'$ is an endpoint, retain $\underline{f_{z'}}$; otherwise repeat this reduction.

    This process terminates. Fix a left-to-right ordering of the diagram of $Y$. If the arrow $w\xrightarrow{\alpha}z$ points right, the nonzero path labelled by $p$ starting at $w$ must travel left, because its first arrow differs from $\alpha$. If $z'$ is internal, the arrow labelled by $\alpha$ entering it is not the last arrow of this path, whose label also differs from $\alpha$. It must therefore enter from the left, and the next reduction again moves strictly left. The opposite orientation gives strict movement to the right. Since the diagram is finite, successive reductions either reach an endpoint or give a zero stable class.

    Consequently, $\sHom_{\varLambda}(X,Y)$ is spanned by the stable classes of those $f_z$ for which $z$ is an endpoint of $Y$. There are at most two such vertices, and hence
    \[
    \mathrm{dim}_{\Bbbk}\sHom_{\varLambda}(X,Y)\leq2.
    \qedhere
    \]
\end{proof}

\medskip
Combining this bound with Proposition~\ref{prop:periodic-stable-Hom} gives the following result. The inequality extends the estimates for Brauer graph algebras in \cite[Propositions 2.2 and 2.8]{A} to the deformed-loop modules considered here.

\begin{Cor}\label{lem:equal-dimension}
Let $\varLambda=\Bbbk Q/I$ be a symmetric stably biserial algebra with the presentation in Theorem~\ref{thm:sym-StBA}. Let $\alpha:v\to v$ be a deformed loop, and let $X=\varLambda e_v/\varLambda\alpha$
be the associated string module, where $e_v$ is the primitive idempotent corresponding to $v$. Then, for every string $\varLambda$-module $Y$, we have
\[
\mathrm{dim}_{\Bbbk}\sHom_{\varLambda}(X,Y)
=\mathrm{dim}_{\Bbbk}\sHom_{\varLambda}(\Omega_{\varLambda}X,Y)
\leq2.
\]
\end{Cor}

We end this section with the following example.  

\begin{Ex1}\label{string module}
Let $\varLambda$ be the symmetric stably biserial algebra in Example~\ref{ex-stably-biserial}. Consider the string modules $X$ and $Y$ with bases and nonzero arrow actions given by
\[
\begin{gathered}
X:\quad x_1\xrightarrow{\alpha}y_1\xrightarrow{\beta}x_2,\\
Y:\quad x_3\xrightarrow{\alpha}y_2\xrightarrow{\beta}x_5
\xleftarrow{\gamma}x_4\xrightarrow{\alpha}y_3.
\end{gathered}
\]
Here each $x_i$ belongs to the component at vertex $1$, each $y_i$ belongs to the component at vertex $2$, and all unspecified arrow actions are zero.

Since $X\cong\varLambda e_1/\varLambda\gamma$, a homomorphism $X\to Y$ is determined by the image of $x_1$, which can be any element of
$\ker(\gamma:e_1Y\to e_1Y)=\Bbbk x_3\oplus\Bbbk x_5$.
Thus $\mathrm{Hom}_{\varLambda}(X,Y)$ has basis $f_1,f_2:X\to Y$, determined by $f_1(x_1)=x_3$, and $f_2(x_1)=x_5$.
Their images are the uniserial submodule spanned by $x_3,y_2,x_5$ and the simple submodule $\Bbbk x_5$, respectively.

The relation $\gamma(\gamma-\beta\alpha)=0$ gives a homomorphism
\[
\varphi:X\longrightarrow\varLambda e_1,\qquad
x_1\longmapsto\gamma-\beta\alpha.
\]
It sends $x_1,y_1,x_2$ to the linearly independent elements $\gamma-\beta\alpha,\alpha\gamma,\gamma^2$, respectively, so $\varphi$ is injective. Since projective $\varLambda$-modules are injective, every map from $X$ to a projective module extends along $\varphi$. Consequently, the maps $X\to Y$ factoring through projective modules are precisely the composites $q\varphi$ with $q:\varLambda e_1\to Y$.

Every such $q$ is determined by $q(e_1)=ax_3+bx_4+cx_5$, with $a,b,c\in\Bbbk$, and the arrow actions give
\[
q\varphi(x_1)
=(\gamma-\beta\alpha)(ax_3+bx_4+cx_5)
=(b-a)x_5.
\]
Hence these composites form $\Bbbk f_2$, and therefore
\[
\mathrm{Hom}_{\varLambda}(X,Y)
=\Bbbk f_1\oplus\Bbbk f_2,
\qquad
\sHom_{\varLambda}(X,Y)=\Bbbk\underline{f_1}.
\]
Moreover, $\Omega_{\varLambda}X\cong\varLambda\gamma$. The projective cover $\varLambda e_1\to\varLambda\gamma$ given by $e_1\mapsto\gamma$ has kernel $\varLambda(\gamma-\beta\alpha)=\varphi(X)$. Thus
$\Omega_{\varLambda}^{2}X\cong X$ and $\top(\Omega_{\varLambda}X)\cong\top(X)$. Together with Proposition~\ref{prop:periodic-stable-Hom}, the above calculation therefore gives
\[
\mathrm{dim}_{\Bbbk}\mathrm{Hom}_{\varLambda}(X,Y)=2,
\qquad
\mathrm{dim}_{\Bbbk}\sHom_{\varLambda}(X,Y)
=\mathrm{dim}_{\Bbbk}\sHom_{\varLambda}(\Omega_{\varLambda}X,Y)=1.
\]
\end{Ex1}

\section{The local one-face case}\label{sec:local-case}

We first compare the Ziegler closures of homogeneous tubes over $B_m(\lambda,\mu)$ and $A_{a,b}$. The Hom-space estimates from Section~\ref{app:lem} identify the possible image tubes under a stable equivalence, and the existence of a common generic module in their closures distinguishes the two algebras when $\max\{a,b\}>2$. We then apply this distinction to the local one-face case needed in Section~\ref{sec:stb-closed}.

\subsection{A common generic module for a family of tubes}
\

Consider the local algebra
\[
\begin{aligned}
B_m(\lambda,\mu)=\Bbbk\langle u,v\rangle/\big(&
(uv)^m-(vu)^m,\ u^2-\lambda(uv)^m,\\
&v^2-\mu(uv)^m,\ (uv)^mu,\ (vu)^mv\big),
\end{aligned}
\]
where $m\geq2$, $\lambda\neq0$ and $\mu\in \Bbbk$. For $c\in \Bbbk$, define
\[
X(c)=B_m(\lambda,\mu)/B_m(\lambda,\mu)\bigl(u-cv(uv)^{m-1}\bigr).
\]
The algebra $B_m(\lambda,\mu)$ has a basis consisting of $1$, the two alternating paths of each length $1,\ldots,2m-1$, and $(uv)^m=(vu)^m$. The linear form taking the coefficient of $(uv)^m$ is a symmetrizing form: it has the same value on $zz'$ and $z'z$ for all $z,z'\in B_m(\lambda,\mu)$, and the complementary alternating paths give a nondegenerate pairing. Thus $B_m(\lambda,\mu)$ is symmetric. The same basis gives $\mathrm{dim}_{\Bbbk} B_m(\lambda,\mu)=4m$ and $\mathrm{dim}_{\Bbbk} X(c)=2m$. In this presentation, $u$ is a deformed loop, and $X(0)$ is the associated string module from Subsection~\ref{subsec:symmetric-stably-biserial}.

\begin{lemma}\label{lem:loc-syzygy}
For every $c\in \Bbbk$,
\[
\Omega_{B_m(\lambda,\mu)}X(c)\cong X(\lambda-c),\qquad
\Omega_{B_m(\lambda,\mu)}^2X(c)\cong X(c).
\]
In particular, the modules $X(0)$ and $X(\lambda)$ are mouths of distinct homogeneous tubes.
\end{lemma}

\begin{proof}
The relations give $uv(uv)^{m-1}=v(uv)^{m-1}u=(uv)^m$ and $\bigl(v(uv)^{m-1}\bigr)^2=0$. Hence
\[
\bigl(u-(\lambda-c)v(uv)^{m-1}\bigr)\bigl(u-cv(uv)^{m-1}\bigr)=\lambda(uv)^m-c(uv)^m-(\lambda-c)(uv)^m=0.
\]
Right multiplication defines a left-module map
\[
\begin{aligned}
B_m(\lambda,\mu)&\longrightarrow B_m(\lambda,\mu)\bigl(u-cv(uv)^{m-1}\bigr),\\
z&\longmapsto z\bigl(u-cv(uv)^{m-1}\bigr),
\end{aligned}
\]
whose kernel contains $B_m(\lambda,\mu)\bigl(u-(\lambda-c)v(uv)^{m-1}\bigr)$. This ideal and the kernel both have dimension $2m$, so the image is isomorphic to $X(\lambda-c)$. Since the image is the kernel of the projective cover of $X(c)$, the syzygy formula follows.

For later use, describe $X(c)$ by a basis $e_0,\ldots,e_{2m-1}$ with
\[
e_0\xrightarrow{v}e_1\xrightarrow{u}e_2
\xrightarrow{v}\cdots\xrightarrow{v}e_{2m-1},
\qquad ue_0=ce_{2m-1},
\]
all unspecified actions being zero. An isomorphism $f:X(c)\to X(c')$ has the same nonzero scalar $a$ on the top and the socle, since $v(uv)^{m-1}e_0=e_{2m-1}$. Comparing the $e_{2m-1}$-components of $f(ue_0)=uf(e_0)$ gives $ac=c'a$. The other nonzero $u$-arrows end at even-numbered layers, whereas $2m-1$ is odd. Thus $c=c'$. The fact that these modules lie at the mouths of homogeneous tubes follows from the explicit almost split sequences in Lemma~\ref{lem:loc-tubes} below. Since $\lambda\neq0$, the modules $X(0)$ and $X(\lambda)$ are nonisomorphic and their tubes are distinct.
\end{proof}

\subsubsection{The modules of every quasi-length}
\

Write a $\Bbbk[t]$-module as $(V,T)$, where $T:V\to V$ is the action of $t$. Define
\[
\begin{gathered}
L:\Bbbk[t]\text{-}\mathrm{Mod}\longrightarrow B_m(\lambda,\mu)\text{-}\mathrm{Mod},\\
L(V,T)=\bigoplus_{i=0}^{2m-1}V_i,
\end{gathered}
\]
where $V_i$ is a labelled copy of $V$, with inclusion $\iota_i:V\to L(V,T)$. The arrow actions are
\[
v\iota_i(z)=\iota_{i+1}(z)
\qquad\text{if }0\le i\le 2m-2\text{ is even},
\]
\[
u\iota_i(z)=\iota_{i+1}(z)
\qquad\text{if }1\le i\le 2m-3\text{ is odd},
\]
\[
u\iota_0(z)=\iota_{2m-1}(Tz),
\]
and all other actions are zero. When the action of $t$ is understood, we also write $L(V)$ for $L(V,T)$. For a $\Bbbk[t]$-linear map $f:(V,T)\to(V',T')$, set $L(f)\iota_i=\iota_i'f$. On these modules, $u^2=v^2=(uv)^m=(vu)^m=0$, so the defining relations of $B_m(\lambda,\mu)$ are satisfied. The functor $L$ is exact and preserves products, direct limits and direct summands.

\begin{lemma}\label{lem:loc-tubes}
Let $J_r(c)$ be one Jordan block of size $r$ with eigenvalue $c$. Then
\[
X_r(c)=L(\Bbbk^r,J_r(c))\qquad(r\geq1)
\]
is the module of quasi-length $r$ in the homogeneous tube $\mathcal{T}(c)$ with mouth $X_1(c)=X(c)$.
\end{lemma}

\begin{proof}
For $c\neq0$, these are the rank-$r$ band matrices for the primitive cyclic word $u^{-1}v(uv)^{m-1}$. For $c=0$, choose a Jordan basis with $J_r(0)e_j=e_{j-1}$ for $j\geq2$ and $J_r(0)e_1=0$. Expanding this basis gives $r$ copies of the alternating chain displayed above, with $u$ mapping the top of the $j$th chain to the bottom of the $(j-1)$st chain. Thus $X_r(0)$ is the string obtained by joining $r$ such chains by backwards $u$-arrows.

The string and band Auslander--Reiten rules of \cite{BR} give
\[
0\longrightarrow X_r(c)\longrightarrow
X_{r-1}(c)\oplus X_{r+1}(c)\longrightarrow X_r(c)\longrightarrow0,
\]
where $X_0(c)=0$. In the string case the two middle terms correspond to deleting the outermost alternating chain and to adding one further chain joined by a backwards $u$; the deletion contributes the zero module when $r=1$. Explicitly, the almost split sequence ending at $X_r(c)$ is obtained by applying $L$ to
\[
0\longrightarrow \frac{\Bbbk[t]}{((t-c)^r)}
\xrightarrow{\alpha_r}
\frac{\Bbbk[t]}{((t-c)^{r-1})}\oplus\frac{\Bbbk[t]}{((t-c)^{r+1})}
\xrightarrow{\beta_r}\frac{\Bbbk[t]}{((t-c)^r)}\longrightarrow0,
\]
where
\[
\begin{aligned}
\alpha_r(\overline z)&=\bigl(\overline z,\overline{(t-c)z}\bigr),\\
\beta_r(\overline a,\overline b)&=\overline{(t-c)a-b}.
\end{aligned}
\]
Here each bar denotes the residue class in the corresponding quotient module.
We identify $\Bbbk[t]/((t-c)^j)$ with its Jordan-block representation, and $\Bbbk[t]/((t-c)^0)=0$.

The calculation is made over the string algebra $B_m(\lambda,\mu)/\bigl((uv)^m\bigr)$ and gives the same almost split sequences over $B_m(\lambda,\mu)$, except possibly for a sequence ending at $B_m(\lambda,\mu)/\soc B_m(\lambda,\mu)$; see \cite[Proposition 4.5]{BR} and the beginning of \cite[\S 4]{AZ1}. That exception cannot occur, since $\mathrm{dim}_{\Bbbk} X_r(c)=2mr$ is even and $\mathrm{dim}_{\Bbbk}\bigl(B_m(\lambda,\mu)/\soc B_m(\lambda,\mu)\bigr)=4m-1$ is odd. This proves the assertion about quasi-length.
\end{proof}

\subsubsection{Construction of the common generic module}
\

\begin{lemma}\label{lem:loc-generic}
Let $t$ act on $\Bbbk(t)$ by multiplication. The module $G=L(\Bbbk(t),t)$ is generic, and
\[
G\in\operatorname{Def}_{B_m(\lambda,\mu)}(\mathcal T(c))
\qquad\text{for every }c\in \Bbbk.
\]
Equivalently,
\[
G\in\bigcap_{c\in \Bbbk}\overline{\mathcal T(c)}^{\,\operatorname{Zg}B_m(\lambda,\mu)}.
\]
In particular, $\mathcal{T}(0)$ and $\mathcal{T}(\lambda)$ have a common generic module in their Ziegler closures.
\end{lemma}

\begin{proof}
Simultaneous multiplication by $a\in \Bbbk(t)$ on each copy of $\Bbbk(t)$ in $G$ commutes with the $B_m(\lambda,\mu)$-action, giving an embedding $\Bbbk(t)\hookrightarrow\operatorname{End}_{B_m(\lambda,\mu)}(G)$. Every submodule of $G$ under this endomorphism-ring action is therefore a $\Bbbk(t)$-subspace. Since $\mathrm{dim}_{\Bbbk(t)}G=2m$, the length of any strictly increasing or decreasing chain of such submodules is at most $2m$. Hence $G$ is endofinite.

Since $B_m(\lambda,\mu)$ is finite-dimensional, $\rad M=\rad(B_m(\lambda,\mu))M$ for every left $B_m(\lambda,\mu)$-module $M$; see \cite[Corollary 15.18]{AF}. The prescribed actions of $u$ and $v$ therefore give
\[
\rad^jG=\bigl(\rad(B_m(\lambda,\mu))\bigr)^jG
=\bigoplus_{i=j}^{2m-1}\iota_i(\Bbbk(t))
\qquad(0\leq j\leq2m).
\]
An endomorphism $f\in\operatorname{End}_{B_m(\lambda,\mu)}(G)$ induces a $\Bbbk$-linear map $f_0:\Bbbk(t)\to \Bbbk(t)$ on $G/\rad G$. Commutation with the successive identity arrows forces the same map $f_0$ on every radical layer, including the bottom layer $\iota_{2m-1}(\Bbbk(t))$. Comparing the bottom components in $f(u\iota_0(z))=uf(\iota_0(z))$ yields
\[
f_0(tz)=t\,f_0(z).
\]
Indeed, the action $u\iota_0(z)=\iota_{2m-1}(tz)$ is the only $u$-arrow ending in $\iota_{2m-1}(\Bbbk(t))$. Thus $f_0$ commutes with all polynomial operators and their nonzero inverses. It is $\Bbbk(t)$-linear, hence multiplication by an element of $\Bbbk(t)$.

Taking the induced action on the top gives a surjective ring homomorphism
\[
\operatorname{End}_{B_m(\lambda,\mu)}(G)\longrightarrow \Bbbk(t),
\]
split by simultaneous scalar multiplication. Every endomorphism in its kernel maps $\rad^jG$ into $\rad^{j+1}G$. Hence every product of $2m$ elements of the kernel is zero. The endomorphism ring is therefore local, since its quotient by this nilpotent ideal is the field $\Bbbk(t)$. In particular, $G$ is indecomposable. Since $\mathrm{dim}_{\Bbbk} \Bbbk(t)=\infty$, $G$ is generic.

Now fix $c\in\Bbbk$, and let $W_c$ and $V_c$ be as in
Example~\ref{ex:Zg-tube}. The construction there shows that
$V_c$ is a direct limit of copies of $W_c$ and that
$\Bbbk(t)$ is a direct summand of $V_c$ as a $\Bbbk[t]$-module.
Since $L$ preserves products,
\[
L(W_c)\cong\prod_{r\geq1}X_r(c)
\in\operatorname{Def}_{B_m(\lambda,\mu)}(\mathcal T(c)).
\]
Preservation of direct limits and direct summands then gives
\[
G=L(\Bbbk(t),t)
\in\operatorname{Def}_{B_m(\lambda,\mu)}(\mathcal T(c)).
\]
Since $G$ is generic, the closure criterion in
Subsection~\ref{subsec:ziegler} yields
$G\in\overline{\mathcal T(c)}^{\,\operatorname{Zg}B_m(\lambda,\mu)}$.
\end{proof}

\subsection{The two families over a local one-face algebra}\label{subsec:local-two-families}
\

The local representation-infinite Brauer graph algebra whose graph has one face has the presentation
\[
A_{a,b}=\Bbbk\langle x,y\rangle/(xy,yx,x^a-y^b)\qquad(a,b\geq2).
\]
For $c\in \Bbbk$, put
\[
\begin{aligned}
Y(c)&=A_{a,b}/(x-cy^{b-1}),\\
Z(c)&=A_{a,b}/(y-cx^{a-1}),
\end{aligned}
\]
and denote their homogeneous tubes by $\mathcal{Y}(c)$ and $\mathcal{Z}(c)$, respectively. These lists include the maximal chain mouths at $c=0$ and the exceptional rank-one band mouths at $c\neq0$.

We specialize Definition~\ref{def:exceptional-band} to $A_{a,b}$. The minimal band lines are $x,y$ and the maximal ones are $x^{a-1},y^{b-1}$. An exceptional primitive band is therefore $x^{-1}y^{b-1}$ or $y^{-1}x^{a-1}$, up to cyclic permutation and inversion. Thus the displayed lists exhaust the relevant mouths. Moreover,
\[
\Omega_{A_{a,b}}Y(c)\cong Z(-c),\qquad \Omega_{A_{a,b}}Z(c)\cong Y(-c).
\]
Indeed, the defining relations give
\[
(x-cy^{b-1})(y+cx^{a-1})=c(x^a-y^b)=0.
\]
Since $A_{a,b}$ is commutative, the reversed product also vanishes. The dimensions $\mathrm{dim}_{\Bbbk} A_{a,b}=a+b$, $\mathrm{dim}_{\Bbbk} Y(c)=b$ and $\mathrm{dim}_{\Bbbk} Z(-c)=a$ identify the kernel of right multiplication by $x-cy^{b-1}$ with $A_{a,b}(y+cx^{a-1})$. This proves the first syzygy formula, and interchanging $x,a$ with $y,b$ proves the second. Hence the syzygy functor interchanges the two families of tubes, with parameter change $c\mapsto-c$.

The two displayed families may overlap when $a=b=2$; no disjointness assertion is needed.

\begin{lemma}\label{lem:local-intersection}
For any $c,c'\in \Bbbk$, if $\mathcal{Y}(c)$ and $\mathcal{Z}(c')$ have a common generic module in their Ziegler closures, then $a=b=2$. Thus, when $\max\{a,b\}>2$, a tube from one family and a tube from the other have no common generic module in their Ziegler closures.
\end{lemma}

\begin{proof}
Note that the quasi-length $r$ module $Y_r(c)$ has $b$ copies $V_0,\ldots,V_{b-1}$ of $\Bbbk^r$, with $y:V_i\to V_{i+1}$ the identity, $x:V_0\to V_{b-1}$ equal to $J_r(c)$, and all other actions zero. For $c\neq0$ these are band matrices. For $c=0$, expanding the Jordan basis gives $r$ maximal $y$-chains joined by backwards $x$-arrows. The same string calculation as in Lemma~\ref{lem:loc-tubes} gives their quasi-lengths. The exceptional module $A_{a,b}/\soc A_{a,b}$ does not intervene: it has one-dimensional top, whereas $Y_r(c)$ has top dimension $r$; for $r=1$ its dimension $b$ is strictly smaller than $a+b-1$.

The matrices give
\[
x^2Y_r(c)=0,\qquad xY_r(c)\subseteq y^{b-1}Y_r(c).
\]
Interchanging $x,a$ with $y,b$ gives the analogous matrices for $Z_r(c')$ and
\[
y^2Z_r(c')=0,\qquad yZ_r(c')\subseteq x^{a-1}Z_r(c').
\]

Indeed, the annihilation condition $x^2N=0$ is preserved by products, direct limits and submodules. The condition $xN\subseteq y^{b-1}N$ means that, for every $z\in N$, the equation
$y^{b-1}w=xz$
has a solution $w\in N$. Products preserve this property by choosing solutions coordinatewise, and direct limits preserve it by choosing a solution at a stage representing $z$.

To verify closure under pure submodules directly, let $N\subseteq E$ be pure and let $z\in N$. Choose $w\in E$ with $y^{b-1}w=xz$. The image $\bar w\in E/N$ is annihilated by $y^{b-1}$ and hence defines an $A_{a,b}$-linear map
\[
A_{a,b}/(y^{b-1})\longrightarrow E/N,\qquad 1\longmapsto\bar w.
\]
By purity this map lifts to $E$. If $w'$ is the image of $1$ under a lift, then $w-w'\in N$ and $y^{b-1}w'=0$. Thus $w-w'$ solves the equation in $N$. It follows that every $N\in\operatorname{Def}_{A_{a,b}}(\mathcal Y(c))$ satisfies $x^2N=0$ and $xN\subseteq y^{b-1}N$. Interchanging $x,a$ with $y,b$ shows that every $N\in\operatorname{Def}_{A_{a,b}}(\mathcal Z(c'))$ satisfies $y^2N=0$ and $yN\subseteq x^{a-1}N$.

Let $H$ be a generic module in both closures. The description of Ziegler closure in Subsection~\ref{subsec:ziegler} places $H$ in both definable subcategories.  Therefore, 
\[
x^2H=y^2H=0,\qquad
xH\subseteq y^{b-1}H,\qquad yH\subseteq x^{a-1}H.
\]
If $a\geq3$, then $x^{a-1}H=0$, so $yH=0$ and then $xH=0$. If $b\geq3$, the same reasoning first gives $xH=0$ and then $yH=0$. In either case $(\rad A_{a,b})H=0$. Hence $H$ is a direct sum of copies of the unique simple module. Its indecomposability would make it one-dimensional, contrary to its being generic. Thus $a=b=2$.
\end{proof}

\subsection{Stable inequivalence and the local one-face case}
\

\begin{Lem}\label{lem:local-not-stably-equivalent}
For $a,b\geq2$ with $\max\{a,b\}>2$, $m\geq2$, $\lambda\in \Bbbk^*$ and $\mu\in \Bbbk$, the algebras $A_{a,b}$ and $B_m(\lambda,\mu)$ are not stably equivalent.
\end{Lem}

\begin{proof}
Suppose that
\[
F:B_m(\lambda,\mu)\text{-}\underline{\mathrm{mod}}\longrightarrow A_{a,b}\text{-}\underline{\mathrm{mod}}
\]
is a stable equivalence. Both algebras are symmetric, so $F\Omega_{B_m(\lambda,\mu)}\cong\Omega_{A_{a,b}}F$ by Proposition~\ref{prop:stable-equivalence-syzygy}. Both algebras are local with nonzero radical square, so Proposition~\ref{prop:stable-equivalence-self-injective} shows that $F$ preserves homogeneous tubes and their mouths. Let $M$ be the indecomposable nonprojective module representing $F(X(0))$. Lemma~\ref{lem:loc-syzygy} therefore shows that $M$ and $\Omega_{A_{a,b}}M$ lie at the mouths of distinct homogeneous tubes.

We first identify these image tubes. If $M$ is a string module, then it is a maximal chain module, as recalled in Subsection~\ref{subsec:symmetric-stably-biserial}. If $M$ is a non-exceptional band module, Proposition~\ref{prop:non-exceptional-tube} gives a string module $Y$ outside the homogeneous tubes such that $\mathrm{dim}_{\Bbbk}\sHom_{A_{a,b}}(M,Y)>2$. The indecomposable nonprojective representative of $F^{-1}(Y)$ also lies outside the homogeneous tubes and is therefore a string $B_m(\lambda,\mu)$-module. Since $X(0)$ is the string module associated with the deformed loop $u$ in the displayed presentation of $B_m(\lambda,\mu)$, Corollary~\ref{lem:equal-dimension} gives the contradiction
\[
2<\mathrm{dim}_{\Bbbk}\sHom_{A_{a,b}}(M,Y)=\mathrm{dim}_{\Bbbk}\sHom_{B_m(\lambda,\mu)}(X(0),F^{-1}(Y))\leq2.
\]
Consequently, $M$ is a maximal chain module or an exceptional band module. The lists in Subsection~\ref{subsec:local-two-families} imply that, after interchanging $x,a$ with $y,b$ if necessary, $M\cong Y(c)$ for some $c\in \Bbbk$. The syzygy formulas, together with preservation of tubes and their mouths, now give
\[
F\mathcal T(0)=\mathcal Y(c),\qquad
F\mathcal T(\lambda)=\mathcal Z(-c).
\]

By Lemma~\ref{lem:loc-generic}, the Ziegler closures of $\mathcal T(0)$ and $\mathcal T(\lambda)$ contain the common generic module $G$. The homeomorphism in Theorem~\ref{thm:Krause-homeomorphism} sends $G$ to a common generic module in the Ziegler closures of $\mathcal Y(c)$ and $\mathcal Z(-c)$. This contradicts Lemma~\ref{lem:local-intersection}, since $\max\{a,b\}>2$.
\end{proof}

We now apply the preceding lemma to the local case of the main theorem.

\begin{Prop}\label{prop:local-one-face}
Let $A$ be a local Brauer graph algebra whose Brauer graph has only one face, and let $B$ be a basic symmetric algebra without semisimple summands. If $A$ and $B$ are stably equivalent, then $B$ is a Brauer graph algebra.
\end{Prop}

\begin{proof}
The representation-finite case follows from the stable classification of symmetric algebras of finite representation type recalled in the introduction; see \cite{ARS}. Assume that $A$ is representation-infinite. Theorem~\ref{thm:AR-conj}, together with the absence of semisimple summands, implies that $B$ is local, and Theorem~\ref{thm:sta-to-BGA=StB} shows that $B$ is stably biserial. Since the one-edge Brauer graph of $A$ has only one face, its edge joins two distinct vertices. Hence $A\cong A_{a,b}$ for some $a,b\geq2$.

Suppose that $B$ is not a Brauer graph algebra. Theorem~\ref{thm:sym-StBA} then gives $\operatorname{char}\Bbbk=2$ and $B\cong B_m(\lambda,\mu)$ with $\lambda\neq0$, after interchanging $u$ and $v$ if necessary. Lemma~\ref{lem:stb-local-iso} excludes $m=1$, so $m\geq2$.

If $a=b=2$, both $A$ and $B$ have nonzero radical square and hence have no nodes by Lemma~\ref{lem:selfinjective-no-nodes}. Proposition~\ref{prop:stable-K0} and the formula for the stable Grothendieck group of a basic local algebra in Subsection~\ref{subsec:stable-K0} yield
\[
\mathbb Z/4\mathbb Z
\cong K_0^{\mathrm{st}}(A_{2,2})
\cong K_0^{\mathrm{st}}(B_m(\lambda,\mu))
\cong\mathbb Z/(4m)\mathbb Z,
\]
which would force $m=1$. Thus $\max\{a,b\}>2$, and Lemma~\ref{lem:local-not-stably-equivalent} contradicts the assumed stable equivalence. Therefore $B$ is a Brauer graph algebra.
\end{proof}

\section{Main results}\label{sec:stb-closed}

Before proceeding to the proof of the main theorems, we establish two more lemmas.

\begin{Lem}\label{lem:stable-tree=BGA}
Let $A$ be a Brauer graph algebra associated with the Brauer graph $(\Gamma,m)$, and let $B$ be a symmetric stably biserial algebra defined by a Brauer graph $(\Gamma',m')$ and a set $\mathcal L$ of deformed loops. If $A$ and $B$ are stably equivalent and $\Gamma$ has only one face, then $B$ is a Brauer graph algebra.
\end{Lem}

\begin{proof}
If $A$ is local, the assertion follows from Proposition~\ref{prop:local-one-face}. We may therefore assume that $A$ is not local. We note that both algebras $A$ and $B$ are indecomposable because we assume that the corresponding Brauer graphs are connected. According to Theorem \ref{thm:sym-StBA}, the conclusion holds when the field characteristic is not $2$, since in that case every symmetric stably biserial algebra is isomorphic to a Brauer graph algebra. Thus, we may restrict to the case where the field $\Bbbk$ has characteristic $2$. We may {also} assume that $A$ is representation-infinite. Then $B$ is also representation-infinite, so neither algebra has a block of Loewy length $2$. Let $n$ be the number of isomorphism classes of simple $A$-modules. Then the unique face of $\Gamma$ has perimeter $2n$.

Since $A$ is not local, $n>1$. By Proposition \ref{prop:stbA-tube-and-face}, the number of faces of perimeter $2n$ in $\Gamma$ (resp.~$\Gamma'$) coincides with the number of tubes of rank $n$ in the Auslander-Reiten quiver of the algebra $A$ (resp.~$B$) that are not stable under the syzygy $\Omega$, divided by $2$. By Proposition~\ref{prop:stable-equivalence-self-injective}, a stable equivalence preserves the ranks of tubes and whether they are invariant under the syzygy functor. Since $\Gamma$ has a face of perimeter $2n$, it follows that $\Gamma'$ must also contain a face of perimeter $2n$.

As $B$ is stably equivalent to $A$, it follows from Theorem \ref{thm:AR-conj} that the number of isomorphism classes of simple $B$-modules is also $n$. Hence, the ribbon graph $\Gamma'$ has $n$ edges, and the sum of the perimeters of all its faces is $2n$. This forces $\Gamma'$ to have exactly one face.

Now, suppose that $B$ is not a Brauer graph algebra. Then the quiver of $B$ must contain deformed loops, implying that $\Gamma'$ contains loops. In that case, $(\Gamma', m')$ would have a face of perimeter $1$, which is a contradiction. We conclude that $\mathcal{L} = \varnothing$ and therefore $B$ is a Brauer graph algebra.

\end{proof}

We prove the following lemma, which will be needed in the proof of our main result.

\begin{Lem}\label{lem:exist-string-then-tree}
Let $A=\Bbbk Q/I$ be a representation-infinite Brauer graph algebra with associated Brauer graph $\Gamma$. {Suppose that} there exist non-isomorphic string modules $M_1$ and $M_2$ such that
\begin{enumerate}
\item $\Omega_A(M_1) = M_2$ and $\Omega_A(M_2) = M_1$;
\item $\mathrm{top}(M_1) \cong \mathrm{top}(M_2)$;
\item $M_1$ and $M_2$ lie on the mouths of distinct tubes of rank~$1$ in the Auslander-Reiten quiver of $A$.
\end{enumerate}
{Then} the ribbon graph $\Gamma$ consists of a single edge connecting two distinct vertices.
\end{Lem}

\begin{proof}
Since the string modules $M_1$ and $M_2$ lie on the mouths of distinct tubes of rank~$1$, they are maximal uniserial (cf. \cite[II.6.2(1)]{E}). Let $P$ be the projective cover of $M_1$. Since $A$ is symmetric and both $M_1$ and $M_2$ are string modules, we have {the isomorphisms of simple modules}
\[
\soc(M_2) \cong \soc(P) \cong \top(P) \cong \top(M_1).
\]
Denote this simple module by $S_v$, corresponding to a vertex $v$ in the quiver $Q$. By symmetry and the assumption $\mathrm{top}(M_1) \cong \mathrm{top}(M_2)$, it follows that
\[
S_v \cong \top(M_1) \cong \soc(M_1) \cong \top(M_2) \cong \soc(M_2).
\]

Since $M_1 \not\cong M_2$, the starting arrow $\alpha$ of the string defining $M_1$ must differ from the starting arrow $\beta$ of the string defining $M_2$. Consequently, the quotient $P / \soc(P)$ is a string module whose corresponding string contains a subword of the form
\[
\cdots \xleftarrow{\beta} v \xleftarrow{\beta} v \xrightarrow{\alpha} \cdots .
\]
Therefore, $\beta$ is a loop. By symmetry, $\alpha$ is also a loop. Since $\beta^2 \neq 0$ and $A$ is special biserial, we must have $\beta \alpha = \alpha \beta = 0$. Hence, the ribbon graph $\Gamma$ consists of a single edge connecting two distinct vertices.
\end{proof}

We now prove the main results of this section.

\begin{Thm}\label{thm:st-BGA}
Let $A$ be a Brauer graph algebra. Then for any basic symmetric algebra $B$ without semisimple summands, if $B$ and $A$ are stably equivalent, then $B$ is also a Brauer graph algebra.
\end{Thm}

\begin{proof}
Without loss of generality, we may assume that $A$ is indecomposable representation-infinite and the field characteristic is $2$. By the assumptions and Theorem~\ref{thm:sta-to-BGA=StB}, $B$ is a symmetric stably biserial algebra. Assume that the Brauer graph associated with $A$ is $(\Gamma, m)$, and let $B$ be a symmetric stably biserial algebra defined by the Brauer graph $(\Gamma', m')$ and a set $\mathcal{L}$ of deformed loops.

{If $B$ is not special biserial, then the set $\mathcal{L}$ is nonempty. So} $\Gamma'$ contains a loop, which corresponds to a deformed loop in $\mathcal{L}$. Let
$$
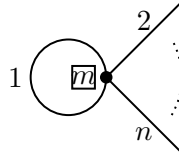
\begin{tikzpicture}
\draw (-0.5,0) circle (0.5);
\fill (0,0) circle (0.5ex);
\node at(-1.2,0) {$1$};
\node at(-0.3,0) {$m$};
\draw[-] (0,0) -- (1,1);
\draw[-] (0,0) -- (1,-1);
\draw (-0.45,-0.15) rectangle (-0.15,0.15);
\node at(0.5,0.75) {$2$};
\node at(0.5,-0.75) {$n$};
\draw[dotted] (0.866,-0.5) arc (-30:30:1);
\end{tikzpicture}$$
be a subgraph of $(\Gamma',m')$, where the edge marked by $1$ gives a deformed loop of $B$. Let
$$\begin{tikzpicture}
\node at(0,0) {$1$};
\node at(0,-0.3) {$2$};
\draw[dotted] (0,-0.5) -- (0,-0.9);
\node at(0,-1.1) {$n$};
\node at(0,-1.4) {$1$};
\node at(0,-1.7) {$1$};
\draw[dotted] (0,-1.9) -- (0,-2.3);
\node at(0,-2.5) {$1$};
\node at(0,-2.8) {$2$};
\draw[dotted] (0,-3) -- (0,-3.4);
\node at(0,-3.6) {$n$};
\node at(0,-3.9) {$1$};
\node at(-1,-2) {$X$};
\node at(-0.5,-2) {$=$};
\end{tikzpicture}$$
be a maximal uniserial $B$-module, then
$$\begin{tikzpicture}
\node at(0,0) {$1$};
\node at(0,-0.3) {$2$};
\draw[dotted] (0,-0.5) -- (0,-0.9);
\node at(0,-1.1) {$n$};
\node at(0,-1.4) {$1$};
\node at(0,-1.7) {$1$};
\draw[dotted] (0,-1.9) -- (0,-2.3);
\node at(0,-2.5) {$1$};
\node at(0,-2.8) {$2$};
\draw[dotted] (0,-3) -- (0,-3.4);
\node at(0,-3.6) {$n$};
\node at(0,-3.9) {$1$};
\node at(-1.5,-2) {$\Omega_{B}(X)$};
\node at(-0.5,-2) {$=$};
\node at(0.6,-2) {$t$};
\draw[dashed] (0.15,-3.9) arc (-15:15:7.5);
\end{tikzpicture},$$
where $t\in \Bbbk^{*}$. {Note that $X$ is a string module and $\Omega_B(X)$ is an exceptional band module, and these two modules lie on the mouths of distinct tubes of rank~$1$ in the Auslander-Reiten quiver of $B$ (cf. Subsection 2.3).} Since $\Omega_B^2X\cong X$ and $\top(\Omega_BX)\cong\top(X)$, Proposition~\ref{prop:periodic-stable-Hom} implies that, for every finite-dimensional $B$-module $Y$,
\[
\mathrm{dim}_{\Bbbk} \sHom_{B}(X, Y) = \mathrm{dim}_{\Bbbk} \sHom_{B}(\Omega_{B}(X), Y).
\]
Let $M$ be the indecomposable nonprojective $A$-module corresponding to $X$ under {a} stable equivalence $\alpha\colon B\text{-}\underline{\mathrm{mod}} \to A\text{-}\underline{\mathrm{mod}}$. By Proposition~\ref{prop:stable-equivalence-syzygy}, $\alpha(\Omega_BX)\cong\Omega_A(M)$. Applying the above identity to a representative of $\alpha^{-1}(S)$ and using full faithfulness of $\alpha$ gives, for every simple $A$-module $S$,
\[
\mathrm{dim}_{\Bbbk} \sHom_{A}(M, S) = \mathrm{dim}_{\Bbbk} \sHom_{A}(\Omega_{A}(M), S).
\]
This equality of dimensions means that the multiplicity of each simple module $S$ in the top of $M$ coincides with its multiplicity in the top of $\Omega_A(M)$. Therefore,
\[
\mathrm{top}(M) \cong \mathrm{top}(\Omega_{A}(M)).
\]

%Since soc$(M)\cong$top$(\Omega_{A}(M))$ and soc$(\Omega_{A}(M))\cong$top$(M)$, we have top$(M)\cong$top$(\Omega_{A}(M))\cong$soc$(M)\cong$soc$(\Omega_{A}(M))$. Since $X$ and $\Omega_{B}(X)$ are exceptional $B$-modules, $M$ and $\Omega_{A}(M)$ are exceptional $A$-modules.

Since an equivalence preserves irreducible morphisms, Proposition~\ref{prop:stable-equivalence-syzygy} shows that $M$ and $\Omega_A(M)$ lie at the mouths of distinct homogeneous tubes.

If $M$ is a string module, then so is $\Omega_A(M)$ since $A$ is a Brauer graph algebra. By Lemma~\ref{lem:exist-string-then-tree}, the Brauer graph $\Gamma$ consists of a single edge connecting two distinct vertices. In particular, $\Gamma$ has only one face. Lemma~\ref{lem:stable-tree=BGA} implies that $B$ is a Brauer graph algebra, contradicting our assumption.

If $M$ is a non-exceptional band module, then by Proposition~\ref{prop:non-exceptional-tube}, for some integer $n > 2$, there exists a string module $Y$ {which does not lie in a tube of rank $1$ in the Auslander-Reiten quiver of $A$} such that
\[
\mathrm{dim}_{\Bbbk}\, \sHom_A(M, Y) > n.
\]
Consequently,
\[
\mathrm{dim}_{\Bbbk}\, \sHom_B(X, \alpha^{-1}(Y)) > n.
\]
Since $Y$ does not lie in a tube of rank $1$ of the Auslander-Reiten quiver of $A$, the same holds for $\alpha^{-1}(Y)$. Therefore, $\alpha^{-1}(Y)$ is a string module. However, by Corollary~\ref{lem:equal-dimension}, we have
\[
\mathrm{dim}_{\Bbbk}\, \sHom_B(X, \alpha^{-1}(Y)) \le 2,
\]
which gives a contradiction.

If $M$ is an exceptional band module, then suppose $M$ is given by the exceptional band
\[
p_{2k}^{-1} p_{2k-1} \cdots p_{2}^{-1} p_{1}
\]
of $A$, with constant $\lambda \in \Bbbk^*$, where $p_1, p_3, \dots, p_{2k-1}$ are minimal band lines and $p_2, p_4, \dots, p_{2k}$ are maximal band lines. By Proposition~\ref{prop:double-faces and exceptional bands}, this exceptional band corresponds to a double-face $F' = (h_1, h_2, \dots, h_k)$ of $\widetilde{\Gamma}$. Suppose $F'$ belongs to the face $F$ of $\widetilde{\Gamma}$. Then the length of $F$ is either $k$ or $2k$.

\textbf{Case 1:} If the length of $F$ is $k$, then $F$ contains exactly the same half-edges as $F'$, and the exceptional band of $\Omega_A(M)$ is also given by the double-face $F'$ of $\widetilde{\Gamma}$. A direct calculation shows that the constant of the band module \(\Omega_A(M)\) is $-\lambda$, which equals $\lambda$ since $\mathrm{char}(\Bbbk) = 2$. Hence $M\cong\Omega_A(M)$, contrary to their lying in distinct homogeneous tubes.

\textbf{Case 2:} If the length of $F$ is $2k$, then $F$ splits into two double-faces $F'$ and $F''$ of $\widetilde{\Gamma}$, and the exceptional band of $\Omega_A(M)$ is given by $F''$. Then the face $F$ of $\widetilde{\Gamma}$ induces an oriented cycle
\[
c = \alpha_{2k} \cdots \alpha_2 \alpha_1
\]
in $Q_{\widetilde{\Gamma}}$. By the definition of a face, all $\alpha_i$ are distinct. The sources of distinct arrows may coincide. We have
\[
\begin{aligned}
\mathrm{top}(M)&=S_{s(\alpha_1)}\oplus S_{s(\alpha_3)}\oplus\cdots\oplus S_{s(\alpha_{2k-1})},\\
\mathrm{top}(\Omega_A(M))&=S_{s(\alpha_2)}\oplus S_{s(\alpha_4)}\oplus\cdots\oplus S_{s(\alpha_{2k})}.
\end{aligned}
\]
Since $\mathrm{top}(M) \cong \mathrm{top}(\Omega_A(M))$, for each vertex $v$ in the oriented cycle $c$, there exist an odd index $2i-1$ and an even index $2j$ such that $v=s(\alpha_{2i-1})=s(\alpha_{2j})$. This implies that there are two arrows in $c$ starting at $v$ (namely $\alpha_{2i-1}$ and $\alpha_{2j}$) and two arrows terminating at $v$ (namely $\alpha_{2i-2}$ and $\alpha_{2j-1}$). Since $Q_{\widetilde{\Gamma}}$ is connected, and each vertex in $Q_{\widetilde{\Gamma}}$ has at most two incoming arrows and at most two outgoing arrows, the oriented cycle $c$ passes through every vertex of $Q_{\widetilde{\Gamma}}$, and every arrow of $Q_{\widetilde{\Gamma}}$ occurs in $c$ exactly once.  Therefore, $F$ is the only face of $\widetilde{\Gamma}$, which means that $\Gamma$ also has only one face (cf. Subsection \ref{subsec:excep-band}). By Lemma~\ref{lem:stable-tree=BGA}, $B$ is a Brauer graph algebra, which contradicts the initial assumption.
\end{proof}

\begin{Thm}\label{thm:st.M-BGA}
Let $A$ be a Brauer graph algebra. Then for any basic algebra $B$ without semisimple summands, if $B$ and $A$ are stably equivalent of Morita type, then $B$ is also a Brauer graph algebra.
\end{Thm}

\begin{proof}
Since $A$ is symmetric, $B$ is symmetric by \cite[Corollary 2.4]{L}. The conclusion follows from Theorem~\ref{thm:st-BGA}.
\end{proof}

The theorem above allows us to give an alternative proof of the main result in \cite{AZ2}.

\begin{Cor} \label{closed-under-derived-equivalence}
Let $A$ be a Brauer graph algebra. Then for any basic algebra $B$, if $B$ and $A$ are derived equivalent, then $B$ is also a Brauer graph algebra.
\end{Cor}

\begin{proof}
Since $A$ is symmetric, it follows from \cite[Corollary 5.3]{Ric} that $B$ is also symmetric. Moreover, it follows from \cite[Corollary 5.5]{Ric} that $A$ and $B$ are stably equivalent of Morita type. Since two derived equivalent algebras have isomorphic centers and (as an algebra) $A$ is indecomposable, $B$ is also indecomposable. Moreover, since $A$ is not a simple algebra, $B$ is also not a simple algebra. Then, by Theorem \ref{thm:st.M-BGA}, $B$ is also a Brauer graph algebra.
\end{proof}

\end{document}